\documentclass[12pt, reqno]{amsart}

\author[M.~Caprio and R.~Gong]{Michele Caprio and Ruobin Gong}
\address{PRECISE Center, Dept. of Computer and Information  Science, 
University of Pennsylvania, 3330 Walnut Street, Philadelphia, PA 19104}
\email{caprio@seas.upenn.edu}
\urladdr{\url{https://mc6034.wixsite.com/caprio}} 
\address{Department of Statistics, Rutgers University, 110 Frelinghuysen Road, Piscataway, NJ}
\email{ruobin.gong@rutgers.edu}
\urladdr{\url{https://statistics.rutgers.edu/people-pages/faculty/people/130-faculty/403-robin-gong}}


\title{Dynamic Precise and Imprecise Probability Kinematics}

\usepackage[T1]{fontenc}
\usepackage{amsmath}
\usepackage{amssymb}
\usepackage{amsthm}
\usepackage[left=2.5cm, right=2.5cm, top=3cm]{geometry}
\usepackage{hyperref}
\usepackage{algorithm}
\usepackage{algpseudocode}
\usepackage{bbm}
\usepackage{tikz}
\usepackage{caption}
\usepackage{subcaption}
\usepackage{fancyhdr}
\usepackage[inline]{enumitem}
\usepackage{comment}
\usepackage{nicefrac}
\usepackage{bm}
\usepackage{mathrsfs}
\usepackage{graphicx}
\usepackage[utf8]{inputenc}
\usepackage{cancel}
\usepackage{mathtools}

\newcommand{\vertiii}[1]{{\left\vert\kern-0.25ex\left\vert\kern-0.25ex\left\vert #1 
    \right\vert\kern-0.25ex\right\vert\kern-0.25ex\right\vert}}
		
\AtBeginDocument{%
   \def\MR#1{}
}

\newtheorem{thm}{Theorem}[section]
\theoremstyle{definition} 
\newtheorem{defi}[thm]{Definition}%[section]
\let\olddefi\defi
\renewcommand{\defi}{\olddefi\normalfont}
\newtheorem{rmk}[thm]{Remark}
\let\oldrmk\rmk
\renewcommand{\rmk}{\oldrmk\normalfont}

\newtheorem{theorem}{Theorem}%[section]
\newtheorem{lemma}[theorem]{Lemma}
\newtheorem{proposition}[theorem]{Proposition}
\newtheorem{corollary}[theorem]{Corollary}

\newtheorem{remark}[theorem]{Remark}
\newtheorem{example}[theorem]{Example}

\hypersetup{
    pdftitle={prove},
    pdfauthor={autori},
    pdfmenubar=false,
    pdffitwindow=true,
    pdfstartview=FitH,
    colorlinks=true,
    linkcolor=blue,
    citecolor=green,
    urlcolor=cyan
}

\providecommand{\MR}[1]{}

\providecommand{\MR}{\relax\ifhmode\unskip\space\fi MR }

\providecommand{\href}[2]{#2}

\newcommand\xqed[1]{%
  \leavevmode\unskip\penalty9999 \hbox{}\nobreak\hfill
  \quad\hbox{#1}}
\newcommand\demo{\xqed{$\triangle$}}

\newcommand{\robin}[1]{{\color{red!70!black}{ \bf \sf \scriptsize Robin:} \sf \scriptsize #1}}
                                 
\begin{document}

\maketitle
\thispagestyle{empty}

\begin{abstract}
    We introduce dynamic probability kinematics (DPK), a method for an agent to mechanically update subjective beliefs in the presence of partial information. We then generalize DPK to dynamic imprecise probability kinematics (DIPK), which allows the agent to express their initial beliefs via a set of probabilities in order to further take ambiguity into account. 
    %take into account the presence of ambiguity. the agent to express their initial beliefs via a set of probabilities. 
    We provide bounds for the lower probability associated with the updated probability sets, and we study the behavior of the latter, in particular contraction, dilation, and sure loss. Examples are provided to illustrate how the methods work.
\end{abstract}

\section{Introduction}\label{intro}
Updating an opinion on the likelihood of an event when new data becomes available is one of the most natural tasks we perform daily. The goal of this paper is to introduce a method to update mechanically the subjective beliefs of an agent that faces ambiguity and who is only able to collect partial information.

%an agent's subjective beliefs in the presence of ambiguity. In particular, we will consider an agent facing ambiguity and collecting partial information. 
With the former, we mean that a single probability measure is not enough to encapsulate the agent's initial beliefs, a very common and well documented situation \cite[Section 1.1.4]{walley}; we inspect ambiguity in Section \ref{impr_prob_sec}. Partial information means that the agent cannot collect crisp evidence; rather, they gather  information whose nature is probabilistic. 
%To model it, we let the agent be unable or unwilling to specify a precise distribution $P_X$ that the observed data comes from. 
Our updating mechanism is based on probability kinematics (PK), an updating rule expressly conceived to deal with partial information. We inspect probability kinematics and its relation with the procedure we present in Section \ref{why}.

We call the method we propose dynamic imprecise probability kinematics (DIPK). It is framed within the credal sets theory paradigm. In this field, a set of probability measures (called a credal set) is used to capture either the ambiguity initially faced by the agent, or inconsistency/imprecision in the process of collecting data. To derive DIPK, we first assume that the agent does not face ambiguity. We come up with a simpler updating technique that we call dynamic probability kinematics (DPK), and then we generalize it by requiring the agent to specify a set $\mathcal{P}$ of probability measures representing their initial beliefs. DIPK is especially useful because it allows the update to be performed mechanically: the agent only needs to specify $\mathcal{P}$. To the best of our knowledge, this is the first time a PK-rooted mechanical procedure to update subjective beliefs in the presence of ambiguity and partial information within the credal sets theory paradigm is presented.

\iffalse
In uncertainty reasoning, probability distributions are the single most widely used mathematical object to encode subjective belief and uncertain evidence. There are situations in which sets of probabilities, rather than a single precise probability, better captures the imprecision occurring when an agent's initial beliefs are elicited and translated. The structural imprecision in a set of probabilities allows for the analysis to be robust against coarsely specified, or potentially misspecified model evidence. 

In this paper, we present a new way to update subjective beliefs based on Jeffrey's rule of conditioning in a dynamic setting. The method can be applied to beliefs in the form of precise probabilities, which we call dynamic probability kinematics (DPK), and can be generalized to work with sets of probabilities, which we call dynamic imprecise probability kinematics (DIPK). 
The contribution of the DPK and DIPK frameworks proposed in this paper hinges on the premise that Jeffrey's rule presents a compelling alternative to Bayes' rule as a reasonable way to update initial beliefs. In this Section, we motivate DPK and DIPK through justifying the importance of its constituent elements. We begin with the Jeffrey's updating rule.
\fi

\subsection{Ambiguity}\label{impr_prob_sec}
%In this work, we consider the dynamic updating of beliefs encoded both in the form of precise and imprecise probabilities. 
Precise probabilities are widely employed as the central vocabulary of many modes of uncertainty reasoning, nearly exclusively so in statistical inference, for example. In the subjective probability literature, the agent's initial beliefs about an event $A \subset \Omega$ are usually  encapsulated in a single probability measure, that is then refined once new information in the form of data become available. As Walley points out in \cite[Section 1.1.4]{walley}, though, missing information and bounded rationality may prevent the agent from assessing probabilities precisely in practice, even if doing so is possible in principle. This may be due to the lack of information on how likely events of interest are,  lack of computational time or ability, or because it is extremely difficult to analyze a complex body of evidence. We call this condition faced by agent \textit{ambiguity} \cite{ellsberg}. Often times agents do not realize they face ambiguity, as observed in \cite{berger} and in the de Finetti lecture delivered at ISBA 2021. There, Berger points out how most people tend to under-report variance; the folklore says by a factor of $3$. People simply think that they know more than they actually do.

%Yet, there are many reasons for an agent to not being able to encapsulate their initial beliefs in a single probability measure \cite[Section 1.1.4]{walley}. In this paper, we work with lower and upper probabilities resulting from specifying sets of probability measures. 

In the presence of ambiguity, the agent may only be able to specify a set $\mathcal{P}$ of probability measures that seem ``plausible'' or ``fit'' to express their initial opinion on the events of interest. 
%evaluate upper and lower bounds of a precise probability. Formally, this can be translated in them specifying a set $\mathcal{P}$ of probabilities. 
Generally speaking, the farther apart (e.g. in the total variation distance) the ``boundary elements'' of $\mathcal{P}$ (i.e. its infimum and supremum), the higher the ambiguity faced by the agent. This way of proceeding, called the \textit{sensitivity analysis approach}, is further examined in Remark \ref{rem_on_lower_prev_and_sensitivity_approach}.

%Another reason for the agent to specify a set $\mathcal{P}$ of probabilities at the beginning of the procedure is to conduct a robust analysis. This is inspected in \cite[Section 1.1.4.(j)]{walley}: the conclusions of a statistical analysis are robust when a realistically wide class of probability models all lead to essentially the same conclusions. To check the robustness of an analysis using a precise probability model, the agent would need to define a ``realistically wide class'' of models by varying those assumptions of the precise model that are arbitrary or unreliable. In effect, this replaces the precise probability model by an imprecise one; conclusions drawn from the imprecise model are automatically robust. A way for the agent to replace the precise probability model with an imprecise one is to start the analysis by specifying a set $\mathcal{P}$ of probabilities.

%Robustness is the main reason we work with sets of probabilities in this paper. As Section \ref{sets} will explain, dynamic probability kinematics may be prior sensitive. To conduct a robust analysis we need to begin with a set of prior probabilities. This leads to the introduction of dynamic imprecise probability kinematics. 

As Section \ref{sets} will discuss, the infima of the sets updated according to our DIPK procedure -- that, as we shall see, are called lower probabilities -- completely characterize the sets. That is why in Section \ref{proc} we give lower and upper bounds for the updated lower and upper probabilities (the conjugate of lower probabilities), respectively, and in Section \ref{behavior} we study the behavior of the updated sets (contraction, dilation, sure loss) by giving sufficient conditions involving lower (and upper) probabilities. 
%So despite DIPK can be considered as an imprecise probability model, we do not consider DIPK lower probabilities as ill-known precise probabilities, as they simply germinate from the agent specifying a set of precise probabilities, the wider the higher their initial uncertainty.

\subsection{Probability kinematics}\label{why}
DPK and DIPK are rooted in probability kinematics (PK), also known as Jeffrey's rule of updating. PK can be seen as a generalization of Bayesian updating, the most famous and widely used technique to describe updating of beliefs. This latter prescribes the scholar to form an initial opinion on the plausibility of the event $A$ of interest, where $A$ is a subset of the state space $\Omega$, and to express it by specifying a probability measure $P$, so that $P(A)$ can be quantified. Once some data $E$ is collected, the Bayesian updating mechanism revises the initial opinion by applying the Bayes' rule
$$P^\star(A) \equiv P(A \mid E) = \frac{P(A \cap E)}{P(E)}=\frac{P(E \mid A) P(A)}{P(E)}\propto P(E \mid A) P(A),$$provided that $P(E) \neq 0$.\footnote{Conditioning on a zero probability event is technically possible, see e.g. literature on lexicographic probability \cite{blume} and layers of zero probabilities \cite{coletti}. We will consider this eventuality in future work, as pointed out in Remark \ref{zero_denom}.} In \cite{jeffrey3,jeffrey,jeffrey2}, Richard Jeffrey makes a compelling case of the fact that Bayes' rule is not the only reasonable way of updating. For example, its use presupposes that both $P(E)$ and $P(A \cap E)$ have been quantified before event $E$ takes place: this can be a very challenging task, for example when $E$ is not anticipated. Jeffrey points out that evidence is not always propositional (i.e. it may not be possible to represent it as a crisp subset); rather, it is oftentimes uncertain or partial.\footnote{Notice that when introducing PK, Jeffrey was not concerned about probabilities being precise: this was one of the main reasons why we introduce DIPK in section \ref{sets}.}  

Bayes' rule is not well-suited for the agent to face partial information. The following example illustrates a situation in which Bayes' rule is not \textit{directly} applicable to compute the updated probability of an event (we would need to enlarge the state space), but Jeffrey's rule can be applied. 

%As we can see,

\begin{example}\label{ex_diac_zab}{\cite[Section 1.1]{diaconis_zabell}}
Three trials of a new surgical procedure are to be conducted at a hospital. Let $1$ denote a successful outcome, and $0$ an unsuccessful one. The state space has the form $\Omega=\{000,001,010,011,100,101,110,111\}$. A colleague informs us that another hospital performed this type of procedure $100$ times, registering $80$ successful outcomes. This information is relevant and should influence our opinion about the outcome of the three trials, but it cannot be put in direct terms  of the occurrence of an event in the original $\Omega$, thus Bayes' rule is not directly applicable. 

%Jeffrey's method of incorporating the information allows to overcome this shortcoming. the probability assignment representing 

Since the description contains no information about the order of the three trials, our initial opinion $P$ assumes that they are exchangeable. That is, consider the partition $\{E_0,E_1,E_2,E_3\}$ of $\Omega$ where $E_j$ is the set of all outcomes with exactly $j$ successes, exchangeability implies that we assign equal probabilities to atomic events within each partition. In other words, $P(\{001\})=P(\{100\})=P(\{010\})$ and $P(\{110\})=P(\{101\})=P(\{011\})$.

%$E_0=\{000\}$, $E_1=\{001,100,010\}$, $E_2=\{110,101,011\}$, and $E_0=\{111\}$. 

The success rate at the other hospital informs our opinion over the partition $\{E_j\}$ only, and nothing more. In relation to our old opinion $P$, our updated opinion $P^\star$ satisfies $P(A \mid E_j)=P^\star(A \mid E_j)$ for all $A \subset \Omega$ and all $j \in \{0,\ldots,3\}$. Upon specifying a new subjective assessment of the $P^\star(E_j)$'s, the updated probability measure $P^\star$ can be fully reassessed by the relation 
$$P^\star(A)=\sum_{j=0}^3 P^\star(A \mid E_j)P^\star(E_j)=\sum_{j=0}^3 P(A \mid E_j)P^\star(E_j).$$
It is within our liberty to reassess the $P^\star(E_j)$'s. We may, for example, regard the three trials as a random subsample of size three from those of the other hospital. This would equate $P^\star(E_j)$ to the probability of obtaining $j$ successes from a Hypergeometric$(100, 80, 3)$ distribution.
% \[
% P^\star(E_j) = Pr_{\text{Hypergeometric(100, 80, 3)}}\left(j \right), \; j \in \{0,\ldots,3\}.
% \]
\demo
\end{example}

The rule $P^\star(A)=\sum_{E_j \in \mathcal{E}} P(A \mid E_j)P^\star(E_j)$ is known as Jeffrey's rule of conditioning. It is valid when there is a partition $\mathcal{E}$ of the state space $\Omega$ such that
\begin{equation}\label{jeffrey_cond}
P^\star( A \mid E_j)=P(A \mid E_j), \quad \forall A \subset \Omega, \forall E_j \in \mathcal{E}.
\end{equation}
As pointed out in \cite[Section 6.11.8]{walley}, under assumption \eqref{jeffrey_cond}, Jeffrey's rule is a consequence of coherence. It is useful when new evidence cannot be identified with the occurrence of an event, but has the effect of changing the probabilities we assign to the events in partition $\mathcal{E}$. It has the practical advantage of reducing the assessment of $P^\star$ to the simpler task of assessing $P^\star(E_j)$, for all $E_j \in \mathcal{E}$. 
In the above example, instead of a full reassessment of probabilities on $\Omega$, the agent only needs to deliberate new assessment of the four probabilities $P^\star(E_0)$ through $P^\star(E_3)$ based on the given information.

To see that Jeffrey's rule of conditioning is a generalization of Bayes' rule, consider partition $\{E,E^c\}$, for some $E \subset \Omega$. Then if $P^\star(E)=1$, we have  that $P^\star(A)=P(A \mid E)P^\star(E)+P(A \mid E^c)P^\star(E^c)=P(A \mid E)$, which is Bayes' rule. In addition, as studied in \cite[Section 2]{diaconis_zabell}, if we are given the couple $\{P,P^\star\}$ of probability measures, we can always reconstruct a partition $\{E_j\}$ for which $\{P,P^\star\}$ could have arisen via Jeffrey's updating rule, unlike Bayesian conditionalization.

Let us now discuss the relation between DPK and Jeffrey's updating. The three main tasks in PK are:
\begin{itemize}
    \item[(1)] Collecting a partition $\mathcal{E}$ of state space $\Omega$;
    \item[(2)] Subjectively assess the probability $P^\star(E)$ to attach to the elements $E$ of partition $\mathcal{E}$;
    \item[(3)] Compute the update $P^\star(A)=\sum_{E\in\mathcal{E}} P(A\mid E)P^\star(E)$.
\end{itemize}
In DPK, we:
\begin{itemize}
    \item[(1')] Collect data points belonging to a generic set $\mathcal{X}$ that induce a partition  $\mathcal{E}$ of state space $\Omega$;
    \item[(2')] Mechanically attach probabilities to the elements of the induced partition;
    \item[(3')] Compute the update as in ``regular'' PK.
\end{itemize}
We allow the evidence observed by the agent to belong to a general set $\mathcal{X}$; data points are regarded as the realization of a random variable $X:\Omega \rightarrow \mathcal{X}$. Notice that if the distribution $P_X$ of $X$ were to be known, the elements of $\mathcal{X}$ would induce a unique partition $\mathcal{E}=\{E_j\}$ of $\Omega$, where $E_j=\{\omega\in\Omega:X(\omega)=x_j\}$ and $P^\star(E_j)=P_X(\{x_j\})$, for all $x_j\in\mathcal{X}$. Instead, to further capture the idea of partial information, we consider the case where $P_X$ is unknown. As we shall see, given data points $x_1,\ldots,x_n \in \mathcal{X}$, they induce a partition $\mathcal{E}=\{E_j\}_{j=1}^{m+1}$, $m\leq n$, where $m$ is the number of unique elements in $\{x_1,\ldots,x_n\}$, $E_j=\{\omega\in\Omega:X(\omega)=x_j\}$ for $j\in\{1,\ldots,m\}$, and $E_{m+1}=(\cup_{j=1}^m E_j)^c$. The relative frequency of $x_1,\ldots,x_n$ will induce the probability that the agent assigns to the elements of $\mathcal{E}$, making the update from $P$ to $P^\star$ mechanical. We inspect subsequent DPK updates in Section \ref{follow}.

%In Sections \ref{sec:1}--\ref{follow} we present a version of Jeffrey's update called dynamic probability kinematics that is easy to work with, and that has a natural application to survey sampling studies, as we show in Section \ref{appl}.

%To this extent, if one considers IP models and does not want to consider them as ill-known precise probabilities

\iffalse
In the same section, we show that, given a generic set $\mathcal{P}^{\text{co}}_{\mathcal{E}_t}$ resulting from DIPK updating of an initial (finite) set $\mathcal{P}$ of probability measures representing the agent's initial beliefs, the pair $\langle f_t, \underline{P}_{\mathcal{E}_t} \rangle$ completely characterize  $\mathcal{P}^{\text{co}}_{\mathcal{E}_t}$. With this we mean that every element in $\mathcal{P}^{\text{co}}_{\mathcal{E}_t}$ can be specified by just knowing the lower probability $\underline{P}_{\mathcal{E}_t}$ associated with $\mathcal{P}^{\text{co}}_{\mathcal{E}_t}$, and the DIPK updating procedure $f_t$. To this extent, even though we work 
\fi

\subsection{Structure of the paper}\label{structure}
The paper is organized as follows. In Section \ref{literature}, we discuss the connection between our work and the existing literature. Sections \ref{sec:1} and \ref{mech} introduce dynamic probability kinematics (DPK). In Section \ref{follow}, we explain how to subsequently update probability measure $P$ as more and more data become available. Section \ref{sets} presents dynamic imprecise probability kinematics (DIPK). In Section \ref{proc}, we give bounds for the upper and lower probabilities associated with the updated probability set, that are then put to use in Section \ref{behavior} to study the behavior of updated sets of probabilities, namely contraction, dilation, and sure loss. Section \ref{appl} presents two examples that illustrate how to implement DPK and DIPK, and Section \ref{concl} concludes our work. Appendix \ref{proofs} contains the proofs of our results.

\section{Related literature}\label{literature}

In this Section, we present some papers that deal with Jeffrey's updating in the context of imprecise probability models. Probability kinematics has been generalized to be put to use in the context of Dempster-Shafer theory, evidence theory, neighborhood models theory, possibility theory, maximum entropy theory, and credal sets theory. DIPK belongs to this last category.

In \cite{shafer}, Shafer discusses Jeffrey's updating from a philosophical perspective, and is the first to consider its application to the context of Dempster-Shafer theory, for which belief functions -- functions representing the degree of belief of the agent on a given event -- and Dempster's updating rule play a central operational role. In \cite{ichihashi} and \cite{smets} the authors further study the generalization of Jeffrey's updating for belief functions defined on a finite state space. In \cite{ichihashi}, the authors point out how Shafer's approach is different from the normative Bayesian approach and is not a straight generalization of Jeffrey's rule, so they propose rules of conditioning for which Jeffrey's rule is a direct consequence of a special case. In \cite{smets}, the author generalizes the results in \cite{ichihashi}. He shows that several forms of Jeffrey's updating rule can be defined so that they correspond to the geometrical rule of conditioning and to Dempster's rule of conditioning, respectively.

In \cite{ma}, the authors provide a generalization of both Jeffrey's rule and Dempster conditioning to propose an effective revision rule in the field of evidence theory. This is very interesting since when one source of evidence is less reliable than another, the idea is to let prior knowledge of an agent be altered only by some of the input information. The change problem is thus intrinsically asymmetric. To this extent, their model takes into account inconsistency between prior and input information. Other works that deal with a generalization of Jeffrey's rule within the framework of evidence theory are \cite{tang1}, in which the authors propose a generalization of probability kinematics where a priori knowledge and new evidence are all modeled by independent random sets, and \cite{tang2} in which a priori knowledge and evidences are modelled by a probability distribution and a
collection of multi-dimensional random sets, respectively.

In \cite{skulj}, the author discusses the application of Jeffrey's rule to neighborhood models theory. In this field, ambiguity is captured by neighborhood of a classical probability measure $P$, presented in the form of interval probabilities $[L,U]$. This means that $P(A) \in [L(A),U(A)]$, for all $A \subset \Omega$, where $\Omega$ is the state space of interest. The author shows that a neighborhood $[L,U]$ of a probability measure $P$ whose lower envelope $L$ is convex or bi-elastic with respect to the base probability measure  \cite[Definitions 3 and 4]{skulj} is closed with respect to Jeffrey’s rule of conditioning. This means that Jeffrey's posterior for $Q \in [L,U]$ still belongs to the interval.

%In the present paper, we are working with a generalization of the \textit{first class of neighbouGammaods} closed with respect to Jeffrey's rule introduced in \cite[Definition 1]{skulj}, by the updating process we describe in Section \ref{update}. 

Possibility theory \cite{zadeh} is a framework alternative to probability theory that is suitable for handling uncertain, imprecise and incomplete knowledge. In possibility theory, there are two different ways to define the conditioning depending on how possibility
degrees are interpreted, one called quantitative possibility and the other called qualitative possibility. In \cite{tabia}, the authors investigate the existence and uniqueness of the posterior probabilities computed according to a possibilistic counterpart of Jeffrey’s rule in both the quantitative and qualitative possibilistic frameworks.

%The principle of maximum entropy (PME) \cite{guiasu} states that the probability distribution which best represents the current state of knowledge about a system is the one with largest entropy, in the context of precisely stated prior data (such as a proposition that expresses testable information). In \cite{lukits}, the author shows that PME generalizes Jeffrey's updating and offers a more integrated approach to probability updating.

In \cite{marchetti}, the authors generalize Jeffrey's rule to credal sets theory. The authors introduce imaginary kinematics \cite[Definition 7]{marchetti}. They combine Jeffrey's rule with Lewis' imaging \cite{lewis} for credal sets to be able to update beliefs when possibly inconsistent probabilistic evidence is gathered. Evidence on some variables is called \textit{inconsistent} when it contradicts certainty (or impossibility) in the agent's knowledge base. There are two main differences between our work and \cite{marchetti}:
\begin{enumerate}
    \item We consider an agent facing ambiguity who specifies a set of probability measures that encapsulates their initial beliefs, while \cite{marchetti} do not;
    %\item In \cite{marchetti}, the agent gathers evidence in the form of elements of $\Omega$, while in this work they may collect information in the form of elements of set $\mathcal{X}$ which can be different from $\Omega$;
    %gather evidence in the form of data points that are realizations of random variable $X:\Omega \rightarrow \mathcal{X}$;
    \item In \cite{marchetti} the authors consider the instance in which gathered evidence is partial and possibly inconsistent, while we only deal with the former.
    %\item In \cite{marchetti} the agent collects new evidence in the form of credal sets, while in the present work they collect data points that are the realization of  $X$  whose distribution is unknown.
\end{enumerate}
In the future we will generalize DIPK by relaxing the (tacit) assumption that the gathered evidence is consistent.

%The main differences between our work and \cite{marchetti} are two. In \cite{marchetti}, the authors require $\Omega$ to be a space of atoms, while we do not put any restriction on the nature of $\Omega$. Also, while we use sets of probability measures in order to capture the impossibility of the agent to specify a single prior probability, and in order to conduct a robust analysis, in \cite{marchetti} the authors use credal sets to  tackle the problem of the agent gathering inconsistent data. In the future, we plan to generalize DIPK to account for imprecise data collection, as we point out in Section \ref{concl}.

It is worth noting that in \cite{ergo_dipk} the authors provide an ergodic theory for the limit of a sequence of successive DIPK updates of a set representing the initial beliefs of an agent. As a consequence, they formulate a strong law of large numbers. Those results are instrumental to increase the applicability of DIPK; for example, they underpin generalizations of classical MCMC procedures that allow for DIPK updating. 

\section{A new way of updating subjective beliefs}\label{sec:1}
%The updating procedure that I'm going to describe here should replace the heavily criticized one we currently have in our paper. It comes from a more attentive read of \cite{diaconis_zabell}. 

In this and in the next Sections, we describe a new way of updating subjective beliefs based on Jeffrey's rule of conditioning \cite{diaconis_zabell,jeffrey3,jeffrey,jeffrey2}, which we call dynamic probability kinematics (DPK). Let $\Omega$ be the state space of interest, and assume it is at most countable. The version of DPK with uncountable $\Omega$ will be the subject of a future work.  Suppose that $P$ is a probability measure on $(\Omega,\mathcal{F})$ representing an agent's initial beliefs around the elements of $\mathcal{F}=2^\Omega$, and that we want to update it after collecting some data.\footnote{We assume $\mathcal{F}=2^\Omega$ to work with the richest possible sigma-algebra; all the results in this paper still hold if $\mathcal{F}$ is not the power set. $\Omega$ is assumed at most countable for simplicity: we want to focus on the updating mechanism and not on measure-theoretic complications.} The agent observes data points $x_1,\ldots,x_n$ that are realizations of a random quantity $X:\Omega \rightarrow \mathcal{X}$ whose distribution is unknown. 
%\footnote{$X$ is a fixed function.} 
Notice that collecting $x_1,\ldots,x_n$ is equivalent to observing $\omega_1,\ldots,\omega_n \sim Q$, where $Q$ is unknown, and then computing $X(\omega_i)=x_i$. Consider now the collection $\mathcal{E}^\prime:=\{E_i\}_{i=1}^n$, where $E_i\equiv X^{-1}(x_i):=\{\omega \in \Omega : X(\omega)=x_i\}$. It induces partition $\mathcal{E}=\{E_j\}_{j=1}^{m+1}$ of $\Omega$, $m\leq n$, whose first $m$ elements are the unique elements of $\mathcal{E}^\prime$, and $E_{m+1}=(\cup_{j=1}^m E_j)^c=\Omega\setminus\cup_{j=1}^m E_j$.

As an update to $P$, we propose
\begin{align}\label{our_upd}
\begin{split}
P_{\mathcal{E}}:\mathcal{F} \rightarrow [0,1] , \quad A \mapsto &P_{\mathcal{E}}(A) :=\sum_{E_j \in \mathcal{E}} P(A \mid E_j) P_{\mathcal{E}}(E_j)\\ \text{such that }& P_{\mathcal{E}}(E_j) \geq 0, \forall E_j \in \mathcal{E} \text{, and } \sum_{E_j \in \mathcal{E}} P_{\mathcal{E}}(E_j)=1.
\end{split}
\end{align}

%We show this is an actual probability measure, that is indeed a Jeffrey's posterior for $P$, and that it meets Jeffrey's condition. 
We have the following.
\begin{proposition}\label{claim1}
$P_{\mathcal{E}}$ is a probability measure, and it is a Jeffrey's posterior for $P$. 
\end{proposition}

In general, Jeffrey's rule of conditioning -- as presented in \cite[Equation 1.1]{diaconis_zabell} -- is given by $P^\star(A)=\sum_j P(A \mid E_j) P^\star (E_j)$, where $P^\star$ is Jeffrey's posterior for $P$. It is valid when Jeffrey's condition is met, that is, when there is a given partition $\{E_j\}$ of the state space $\Omega$ such that $P(A \mid E_j)=P^\star(A \mid E_j)$ is true for all $A \in \mathcal{F}$ and all $j$.  Specifically, this condition is met by $P_{\mathcal{E}}$. Since $P_{\mathcal{E}}$ is a probability measure by Proposition \ref{claim1}, it is true that, for all $A \in \mathcal{F}$, $P_{\mathcal{E}}(A)=\sum_{E_j \in \mathcal{E}} P_{\mathcal{E}}(A \mid E_j)P_{\mathcal{E}}(E_j)$. But given our definition for $P_{\mathcal{E}}$, we also have that $P_{\mathcal{E}}(A)=\sum_{E_j \in \mathcal{E}} P(A \mid E_j)P_{\mathcal{E}}(E_j)$. This implies that there is a partition $\mathcal{E}$ for which $P(A \mid E_j)=P_\mathcal{E} (A \mid E_j)$ is true for all $A \in \mathcal{F}$ and all $E_j \in \mathcal{E}$. 

%In the next Section, we show how computing $P_{\mathcal{E}}(A)$ can be a mechanical procedure that does not require subjective assessment.

%We now show that $P_{\mathcal{E}}$ is actually a Jeffrey's posterior for $P$.

%. The first one is that $P(E_j)>0$, for all $E_j \in \mathcal{E}$, and the second one is that $\cup_{i=1}^n x_i \subsetneq \mathcal{X}$, where $\mathcal{X}$ is the codomain of $X_1,\ldots,x_n$. The first one is needed for the ratio $\frac{P_\mathcal{E}({E}_\omega)}{P({E}_\omega)}$ to be well defined for all $\omega \in \Omega$. The second one is needed because if $\cup_{i=1}^n x_i = \mathcal{X}$, then $E_{t+1}=(\cup_{j=1}^n E_j)^c=\emptyset$, and so $P(E_{t+1})$ would need to be equal to $0$. We are basically asking that the first set of data we observe does not correspond to the whole codomain of the random variables, and that our prior gives a positive probability to all the elements of partition $\mathcal{E}$. If we remove the second assumption, we may have that $\cup_{i=1}^n x_i = \mathcal{X}$. If that is the case, then we cannot refine 

%Since $P_\mathcal{E}$ is a well defined probability measure (by Proposition \ref{claim1}), we have that $P_\mathcal{E}(\{\omega\}) \in [0,1]$, for all $\omega \in \Omega$, which implies that $\frac{P_\mathcal{E}({E}_\omega)}{P({E}_\omega)}P(\{\omega\}) \in [0,1]$, for all $\omega \in \Omega$. In turn, this implies that $P({E}_\omega) \neq 0$, for all $\omega \in \Omega$.

\section{Computing $P_\mathcal{E}$ via an empirical specification}\label{mech}
In this Section, we show how to compute DPK updating for $P_{\mathcal{E}}(A)$ via an empirically specified sequence of partitions, which in turn determines a sequence of empirical probability measures. Utilizing it eases the analyst of the burden of making a full subjective probabilistic assessment for the elements of $\mathcal{E}$.

Recall that $\mathcal{E}^\prime=\{E_i\}_{i=1}^n=\{X^{-1}(x_i)\}_{i=1}^n$, and $\mathcal{E}=\{E_j\}_{j=1}^{m+1}$, where $E_1,\ldots,E_m$ are the unique elements of $\mathcal{E}^\prime$, and $E_{m+1}=(\cup_{j=1}^m E_j)^c$. Denote by $\Delta(\Omega,\mathcal{F})$ the set of all probability measures on $(\Omega,\mathcal{F})$. Then, consider the empirical probability measure $P^{emp}\in\Delta(\Omega,\mathcal{F})$ such that, if $E_{m+1}\neq\emptyset$,
\begin{equation}\label{emp_prob_1}
    P^{emp}(E_j)=\frac{1}{n+1} \sum_{i=1}^{n} \mathbb{I}(E_j=E_i), \quad \text{for all } j \in \{1,\ldots,m\},
\end{equation}
where $\mathbb{I}$ denotes the indicator function, and 
\begin{equation}\label{emp_prob_2}
   P^{emp}(E_{m+1})=1-\sum_{j=1}^{m} P^{emp}(E_j).
\end{equation}
If instead $E_{m+1}=\emptyset$,
\begin{equation}\label{emp_prob_4}
    P^{emp}(E_j)=\frac{1}{n} \sum_{i=1}^{n} \mathbb{I}(E_j=E_i), \quad \text{for all } j \in \{1,\ldots,m\}
\end{equation}
and 
\begin{equation}\label{emp_prob_5}
   P^{emp}(E_{m+1})=0.
\end{equation}
We require that
\begin{equation}\label{emp_prob_3}
    P_{\mathcal{E}}(E_j)=\beta(n) P(E_j) + \left[ 1-\beta(n) \right] P^{emp}(E_j), \quad \forall E_j\in\mathcal{E},
\end{equation}
where $\beta(n)$ is a coefficient in $[0,1]$ depending on $n$: the posterior probability $P_\mathcal{E}$ assigned to the elements $E_j$ of partition $\mathcal{E}$ is a weighted average of the prior $P$ and the empirical probability measure $P^{emp}$. Performing the update in \eqref{our_upd} then becomes a mechanical procedure, making Jeffrey's updating procedure easier to carry out.  DPK is driven by the coefficient $\beta(n)$, which is specified by the agent and controls the extent of prior-data tradeoff in the updated belief. The closer $\beta(n)$ is to $1$, the ``stickier'' DPK is; that is, the less the collected observations influence the agent's (revised) beliefs, and vice versa the closer $\beta(n)$ is to $0$. The facts that DPK is mechanical and that its stickiness is regulated by a parameter that is entirely under the agent's control makes our updating procedure mathematically and conceptually appealing.

\iffalse
Performing the update in \eqref{our_upd} becomes then a mechanical procedure, making Jeffrey's updating procedure easier.  
%(the analyst does not have to subjectively specify the probability assigned to all the elements of $\mathcal{E}$, a task that can be mentally and mathematically challenging). 
DPK is driven by coefficient $\beta(n)$, which is under the control of the agent. The closer it is to $1$, the ``stickier'' DPK is, that is, the less the collected observations influence the agent's (revised) beliefs, and vice versa the closer $\beta(n)$ is to $0$. The facts that DPK is mechanical and that its stickiness is regulated by a parameter that is entirely under the agent's control makes our updating procedure mathematically and conceptually appealing. 
\fi

In the remainder of this paper, we are going to use the procedure we just described to assign updated probabilities to the elements of $\mathcal{E}$. An example of how to update subjective beliefs according to DPK is given in section \ref{surgery}.

\begin{remark}\label{subtle}
There is a subtlety in moving from $P$ to $P_\mathcal{E}$. Let $\breve{P}:=\beta(n) P + [1-\beta(n)] P^{emp}$. Requiring that $P_\mathcal{E}(E_j)=\beta(n) P(E_j) + [1-\beta(n)] P^{emp}(E_j)$, for all $E_j\in\mathcal{E}$, means that the restriction $ P_\mathcal{E} |_{\sigma(\mathcal{E})}$ of $P_\mathcal{E}$ agrees with the restriction $ \breve{P} |_{\sigma(\mathcal{E})}$ of $\breve{P}$ on the sigma algebra $\sigma(\mathcal{E})$ generated by the elements of $\mathcal{E}$. $P_\mathcal{E} |_{\sigma(\mathcal{E})}$ is then extended to $(\Omega,\mathcal{F})$ through $P(\cdot \mid E_j)$, for all $j$, via 
$$P_\mathcal{E}(A)=\sum_{E_j \in\mathcal{E}}P(A \mid E_j)P_\mathcal{E}(E_j).$$
\end{remark}

\iffalse
\robin{Is this a consequence of $X^{-1}$, when regarded as a set-valued map into $\mathcal{E}$, never maps into an empty set?}

{\color{blue} Michele: Yes. Since we are considering $x$'s in $\mathcal{X}$ -- that we defined as the image under $X$ of the domain of $X$ -- we have that $P_X(x)>0$, for all $x \in \mathcal{X}$. Since $E_j=X^{-1}(x_j)$, and $P_\mathcal{E}(E_j)=Q(E_j)=P_X(x_j)$, we then conclude that the only way $P_\mathcal{E}$ can give probability $0$ to some $E_j$ is for that $E_j$ to be the empty set.}
\fi

\section{Subsequent updates}\label{follow}
Let us denote the amount of data available at time $t=1$ by $n_1$. Once at time $t=2$ we observe new data points $x_{n_1+1},\ldots,x_{n_2}$, we update $P_\mathcal{E} \equiv P_{\mathcal{E}_1}$ to $P_{\mathcal{E}_1 \mathcal{E}_2}$ via the same mechanical procedure depicted in Section \ref{mech}. With this, we mean the following. We now have observed data $x_1,\ldots x_{n_1}, x_{n_1+1},\ldots,x_{n_2}$. Then, we consider partition $\mathcal{E}_2=\{E_j\}_{j=1}^{k+1}$, where $E_1,\ldots,E_k$ are the unique elements in the collection $\mathcal{E}^{\prime\prime}=\{E_i\}_{i=1}^{n_2}=\{X^{-1}(x_i)\}_{i=1}^{n_2}$, and $E_{k+1}=(\cup_{j=1}^k E_k)^c$. We equate $P_{\mathcal{E}_1 \mathcal{E}_2}(E_j)=\beta(n_2) P_{\mathcal{E}_1}(E_j) + [1-\beta(n_2)] P^{emp}_2(E_j)$, for all $E_j\in\mathcal{E}_2$, where the $P^{emp}_2(E_j)$'s are computed similarly to \eqref{emp_prob_1}--\eqref{emp_prob_5}, so we have
$$P_{\mathcal{E}_1 \mathcal{E}_2}(A)=\sum_{E_j \in \mathcal{E}_2} P_{\mathcal{E}_1}(A \mid E_j) P_{\mathcal{E}_1 \mathcal{E}_2}(E_j).$$
Clearly, Proposition \ref{claim1} is true also for $P_{\mathcal{E}_1 \mathcal{E}_2}$.

%. In particular, $E_j=\{\omega \in \Omega : X(\omega)=x_j\}$, $j \in \{1,\ldots,m\}$, and $E_{m+1}=(\cup_{j=1}^m E_j)^c$. Also, we assign probability $P_{\mathcal{E}_2}(E_j)=P_X(\{x_j\})$, for $j \in \{1,\ldots,m\}$, and $P_{\mathcal{E}_2}(E_{m+1})=1-\sum_{j=1}^m P_{\mathcal{E}_2}(E_{j})$. 

%If we keep on collecting observations, eventually our sequence of DPK updates converges. 
Call $ (P_{\mathcal{E}_1 \cdots \mathcal{E}_t})$ the sequence of successive updates of probability measure $P$ representing the initial subjective beliefs of the agent around the elements of $\Omega$, and $\mathbf{x}_t=\{x_i\}_{i=1}^{n_t}$ the collection of data points available at time $t$. Notice that 
$$\# \mathcal{E}_t=\#\text{unique}(\mathbf{x}_t)+1,$$
where $\#$ denotes the cardinality operator. That is, the number of elements of partition $\mathcal{E}_t$ is a function of the collected observations up to time $t$; in particular, it is equal to the number of unique observations $x_i$ plus $1$, the complementary of the union of the other elements of $\mathcal{E}_t$. In the remainder of the paper, for notational convenience, we write $P_{\mathcal{E}_t}$ in place of $P_{\mathcal{E}_1 \cdots \mathcal{E}_t}$, for all $t\in\mathbb{N}$.

\begin{remark}\label{reltionship_n_and_t}
Notice that, for all $t\in\mathbb{N}$, $n_t>n_{t-1}$, and $n_0=0$. That is, the amount of data points available at time $t$ is always larger than that at time $t-1$;  
%This means that we cannot lose data points as we proceed in our analysis; it also 
this implies that as $t\rightarrow\infty$, then $n_t\rightarrow\infty$. In addition, we have that $P_{\mathcal{E}_t}$ depends on $n_1,\ldots,n_t$ and $P_{\mathcal{E}_0}$; we denote this by $P_{\mathcal{E}_t} \equiv P_{\mathcal{E}_t}(n_1,\ldots,n_t,P_0)$. To show this, we write $P_{\mathcal{E}_2}$ in terms of $n_1$, $n_2$, and $P_0$. We assume that $E_{k+1}\neq \emptyset$, so the following holds
\begin{align}\label{how_2_deps_on_prior}
\begin{split}
    P_{\mathcal{E}_2}(A)=\sum_{E\in\mathcal{E}_2} &P_{\mathcal{E}_1}(A\mid E) \left[ \beta(n_2) P_{\mathcal{E}_1}(E) + (1-\beta(n_2)) \frac{1}{n_2+1} \sum_{s=1}^{n_2} \mathbb{I}(\tilde{E}_s=E) \right]\\
    =\sum_{E\in\mathcal{E}_2} &\Bigg\{ \frac{\sum_{E^\prime\in\mathcal{E}_1} P_{\mathcal{E}_0}(A\cap E \mid E^\prime) \left[ \beta(n_1) P_{\mathcal{E}_0}(E^\prime) + (1-\beta(n_1)) \frac{1}{n_1+1} \sum_{i=1}^{n_1} \mathbb{I}(\check{E}_i=E^\prime) \right]}{\sum_{E^\prime\in\mathcal{E}_1} P_{\mathcal{E}_0}(E \mid E^\prime) \left[ \beta(n_1) P_{\mathcal{E}_0}(E^\prime) + (1-\beta(n_1)) \frac{1}{n_1+1} \sum_{i=1}^{n_1} \mathbb{I}(\check{E}_i=E^\prime) \right]}\\
    &\cdot \left[ \beta(n_2) \sum_{E^\prime\in\mathcal{E}_1} P_{\mathcal{E}_0}(E \mid E^\prime) \left( \beta(n_1) P_{\mathcal{E}_0}(E^\prime) + (1-\beta(n_1)) \frac{1}{n_1+1} \sum_{i=1}^{n_1} \mathbb{I}(\check{E}_i=E^\prime) \right) \right.\\
    &\left.+ (1-\beta(n_2)) \frac{1}{n_2+1} \sum_{s=1}^{n_2} \mathbb{I}(\tilde{E}_s=E) \right] \Bigg\} \text{, } \check{E}_i\in\mathcal{E}^\prime, \tilde{E}_s \in \mathcal{E}^{\prime\prime}.
\end{split}
\end{align}
It is easy to see how this can be generalized to any $t>2$. In the remainder of the paper, for notational convenience we write $P_{\mathcal{E}_t}$ in place of $P_{\mathcal{E}_1\cdots\mathcal{E}_t}(n_1,\ldots,n_t,P_0)$, and $P_{\mathcal{E}_t}(A)$ in place of $P_{\mathcal{E}_1\cdots\mathcal{E}_t}(n_1,\ldots,n_t,P_0;A)$, for all $A\in\mathcal{F}$.
\end{remark}

A consequence of how we build partitions is that, for any $t$, $\mathcal{E}_t$ is not coarser than $\mathcal{E}_{t-1}$. To see this, suppose $\mathcal{E}_{t-1}$ has $\ell+1$ many elements, that is, $\mathcal{E}_{t-1}=\{E_1^{\mathcal{E}_{t-1}},\ldots,E_\ell^{\mathcal{E}_{t-1}},E_{\ell+1}^{\mathcal{E}_{t-1}}\}$. As we know, this means that $E_{\ell+1}^{\mathcal{E}_{t-1}}=(\cup_{j=1}^\ell E_j^{\mathcal{E}_{t-1}})^c$. Now suppose that in the next updating step we only observe one element $x$. If it is not a ``novelty'', then $\mathcal{E}_t=\mathcal{E}_{t-1}$. If instead $x$ is a new element, we have that $\mathcal{E}_{t}$ has $\ell+2$ many elements. In particular, $E_j^{\mathcal{E}_{t-1}}=E_j^{\mathcal{E}_{t}}$, for all $j \in \{1,\ldots,\ell\}$, and $E_{\ell+1}^{\mathcal{E}_{t-1}}=E_{\ell+1}^{\mathcal{E}_{t}} \cup E_{\ell+2}^{\mathcal{E}_{t}}$. Of course, if we observe more elements, we further refine $E_{\ell+1}^{\mathcal{E}_{t-1}}$. 
%Call now $\tilde{\mathcal{E}}$ the partition resulting from our updating process having possibly countably many elements, and such that $\emptyset \in \tilde{\mathcal{E}}$. Then, we have the following.
\begin{proposition}\label{limit}
There exists a partition $\tilde{\mathcal{E}}$ that cannot be refined as a result of the updating process described in Sections \ref{sec:1} and \ref{mech}.
\end{proposition}

We now show how, under mild standard assumptions, the sequence of successive subjective beliefs updated according to the DPK procedure converges. Call $Q_{\tilde{\mathcal{E}}}$ the restriction of probability measure $Q$ introduced in Section \ref{sec:1} to the sigma algebra $\sigma(\tilde{\mathcal{E}})$ generated by the elements of $\tilde{\mathcal{E}}$. That is, $Q_{\tilde{\mathcal{E}}}:=\left. Q \right|_{\sigma(\tilde{\mathcal{E}})}$, $Q_{\tilde{\mathcal{E}}}:\sigma(\tilde{\mathcal{E}}) \rightarrow [0,1]$. Call then $\mathscr{Q}$ the collection of extensions of $Q_{\tilde{\mathcal{E}}}$ from $\sigma(\tilde{\mathcal{E}})$ to $\mathcal{F}=2^\Omega$. Notice that $\mathscr{Q} \neq \emptyset$ and that $\mathscr{Q}$ is a singleton if and only if $\overline{\mathcal{F}}^{Q_{\tilde{\mathcal{E}}}}=\mathcal{F}$, where
$$\overline{\mathcal{F}}^{Q_{\tilde{\mathcal{E}}}}:=\left\lbrace{ A\in 2^\Omega : Q_{\tilde{\mathcal{E}} \star}(A)=Q_{\tilde{\mathcal{E}}}^\star(A) }\right\rbrace$$
is the $Q_{\tilde{\mathcal{E}}}$-completion of $\sigma(\tilde{\mathcal{E}})$, and $Q_{\tilde{\mathcal{E}} \star}$ and $Q^\star_{\tilde{\mathcal{E}}}$ are the inner and outer measures induced by $Q_{\tilde{\mathcal{E}}}$, respectively. Recall that the total variation distance $d_{TV}$ is defined as
$$d_{TV}(\pi,\gamma):=\sup_{A\in\mathcal{F}} \left| \pi(A)-\gamma(A) \right|,$$
for all $\pi,\gamma\in\Delta(\Omega,\mathcal{F})$.

\begin{theorem}\label{prop3}
If $\mathbb{E}(X)<\infty$, $\lim_{n_t\rightarrow\infty} \beta(n_t)=0$, and $[1-\beta(n_t)]/n_t = O({1}/{n_t})$, then $P_{\mathcal{E}_t}$ converges to an element of $\mathscr{Q}$ with probability $1$ as $n_t \rightarrow \infty$ in the total variation distance.
\end{theorem}

\iffalse
\robin{This is good and indeed needed. Is it correct to say that the rates at which the probabilities converge for the different partition elements $E_j$'s would be different, presumably slower if the $x_j$ corresponding to that $E_j$ has very small/negligible $P_X$ probability?}

{\color{blue} Michele: I think you are right. The smaller the probability $P_X$ assigns to the $x_j$'s, the slower the convergence of $P_\mathcal{E}$. For example, if $X_1,\ldots,x_n \sim Po(\lambda)$ for some $\lambda$, we should have a slower convergence than the case in which $X_1,\ldots,x_n \sim Bin(n,p)$, for some $t$ and some $p$. We may be able to show that the speed of convergence is a decreasing function of the cardinality of $\mathcal{X}$: the higher $\#\mathcal{X}$, the slower the convergence.}
\fi

Because as $n_t$ grows to infinity the partition induced by collection $\{X^{-1}(x_i)\}_{i=1}^{n_t}$ approaches $\tilde{\mathcal{E}}$, we denote by $P_{\tilde{\mathcal{E}}}$ the limit we find in Theorem \ref{prop3}.

\begin{remark}\label{zero_denom}
We tacitly assumed that for all nonempty $A\in\mathcal{F}$, the probability assigned to $A$ by $P$ (representing the agent's initial beliefs) is positive. In formulas, 
\begin{equation}\label{tacit_assumption}
    P(A)>0 , \quad \text{ for all } \emptyset\neq A \in\mathcal{F}.
\end{equation}
This assumption is not too stringent. For example, suppose the agent specifies $P$ so that there is a collection of sets $\{A^\prime_k\}\subset \mathcal{F}$ such that (i) $A^\prime_k\neq\emptyset$ and $P(A^\prime_k)=0$, for all $k$, and (ii) set $\mathcal{A}^\prime:=\cup_k A^\prime_k$ is finite. Then, the agent should choose $\tilde{P}=(1-\epsilon)P+\epsilon U$ as probability encapsulating their initial beliefs, where $U$ is a uniform on all elements with zero atomic probabilities -- that is, a uniform on  $\mathcal{A}^\prime$ -- and $\epsilon$ is an arbitrarily small element of $(0,1)$. This procedure -- a particular case of $\epsilon$-contamination \cite{berger4,berger5,huber3,huber} where the contaminating distribution is a uniform, sometimes referred to as \textit{padding} \cite{padding} -- keeps the initial beliefs essentially unaltered, and avoids complications coming from conditioning on zero probability events. 
In the future, we plan to deal with the delicate matter of conditioning on zero probability events in a more sophisticated way, possibly using techniques from the literature on lexicographic probabilities \cite{blume} or layers of zero probabilities \cite{coletti}.

\end{remark}

\begin{remark}\label{commutativity}
Dynamic probability kinematics is not commutative. With this we mean the following. Consider an initial probability $P$ and compute its dynamic probability kinematics update $P_{\mathcal{E}_1}$ based on partition $\mathcal{E}_1$; then compute the DPK update of $P_{\mathcal{E}_1}$ based on partition $\mathcal{E}_2$, and call this update $P_{\mathcal{E}_1\mathcal{E}_2}$. If we proceed in the opposite direction, that is, if we first update $P$ to $P_{\mathcal{E}_2}$, and then update this latter to $P_{\mathcal{E}_2\mathcal{E}_1}$, we have that, in general, $P_{\mathcal{E}_1\mathcal{E}_2} \neq P_{\mathcal{E}_2\mathcal{E}_1}$. To see this, consider the following scenario.  Let $\mathcal{E}_1$ be the partition induced by observations $x_1,\ldots,x_{n_1}$, and $\mathcal{E}_2$  the partition induced by observations $x_1,\ldots,x_{n_1},x_{n_2=n_1+1}$, where $x_{n_2}=x_{n_1}$. This means that $\mathcal{E}_1=\mathcal{E}_2$, but $P_{\mathcal{E}_1\mathcal{E}_2}(E)\neq P_{\mathcal{E}_2\mathcal{E}_1}(E)$, for all $E\in \mathcal{E}_1=\mathcal{E}_2$.  To illustrate this, let $\#\mathcal{E}_1=\#\mathcal{E}_2=m+1$, $m\leq n_1$, assume $E_{m+1}\neq\emptyset$, and notice that
\begin{equation}\label{not_comm_1}
    P_{\mathcal{E}_1\mathcal{E}_2}(E_j)=\beta(n_1+1) P_{\mathcal{E}_1}(E_j) + \frac{1-\beta(n_1+1)}{n_1+2}\sum_{i=1}^{n_1+1}\mathbb{I}(E_j=X^{-1}(x_i)), \quad j \in \{1,\ldots,m\}
\end{equation}
and 
\begin{equation}\label{not_comm_2}
    P_{\mathcal{E}_1\mathcal{E}_2}(E_{m+1})=1-\sum_{j=1}^m P_{\mathcal{E}_1\mathcal{E}_2}(E_j).
\end{equation}
Instead, suppose that we first update according to $\mathcal{E}_2$ and then according to $\mathcal{E}_1$. This may happen if we lose data point $x_{n+1}$, for example because of a transcription error. Then we have
\begin{equation}\label{not_comm_3}
    P_{\mathcal{E}_2\mathcal{E}_1}(E_j)=\beta(n_1) P_{\mathcal{E}_1}(E_j) + \frac{1-\beta(n_1)}{n_1+1}\sum_{i=1}^{n_1}\mathbb{I}(E_j=X^{-1}(x_i)), \quad j \in \{1,\ldots,m\}
\end{equation}
and 
\begin{equation}\label{not_comm_4}
    P_{\mathcal{E}_2\mathcal{E}_1}(E_{m+1})=1-\sum_{j=1}^m P_{\mathcal{E}_2\mathcal{E}_1}(E_j).
\end{equation}
As we can see, $P_{\mathcal{E}_1\mathcal{E}_2}(E_j) \neq P_{\mathcal{E}_2\mathcal{E}_1}(E_j)$, $j \in \{1,\ldots,m\}$, and $P_{\mathcal{E}_1\mathcal{E}_2}(E_{m+1})\neq P_{\mathcal{E}_2\mathcal{E}_1}(E_{m+1})$. 

\iffalse
We can give sufficient conditions for DPK to be commutative.
\begin{proposition}\label{comm_suff_cond}
For all $t,k\in\mathbb{N}$, consider two partitions $\mathcal{E}_t$ and $\mathcal{E}_{t+k}$ resulting from the updating process described in Sections \ref{sec:1} and \ref{mech}. If, once collected, data points cannot be lost and if $P(A\cap E)=P_{\mathcal{E}_t}(A\cap E)$, for all $A\in\mathcal{F}$ and all $E\in\mathcal{E}_{t+k}$, then $P_{\mathcal{E}_t \mathcal{E}_{t+k}}=P_{\mathcal{E}_{t+k}\mathcal{E}_t}$.
\end{proposition}
The assumption that data points cannot be lost is a common restriction on the data acquisition technique. The fact that $P(A\cap E)=P_{\mathcal{E}_t}(A\cap E)$ for all $A\in\mathcal{F}$ and all $E\in\mathcal{E}_{t+k}$ is telling us that the sole driver of the  updating of beliefs is the empirical probability $P^{emp}_{t+k}$.

\begin{corollary}\label{comm_suff_cond_cor}
If the assumptions in Proposition \ref{comm_suff_cond} hold, then we can write $P_{\tilde{\mathcal{E}}}$ explicitly. In particular,
$$P_{\tilde{\mathcal{E}}}(A)=\sum_{E\in\tilde{\mathcal{E}}}P(A\mid E)Q(E), \quad \forall A\in\mathcal{F}.$$
\end{corollary}
\fi

In \cite[Section 3]{diaconis_zabell}, the authors study when Jeffrey's update is commutative. As we shall see, their results cannot be directly applied to DPK. In \cite[Theorem 3.1]{diaconis_zabell}, the authors show that, given two generic  partitions $\mathcal{E}$ and $\mathcal{G}$, if
\begin{equation}\label{suff_cond_d&z}
    P_{\mathcal{E}\mathcal{G}}(E)=P_{\mathcal{E}}(E) \quad \text{and} \quad P_{\mathcal{G}\mathcal{E}}(G)=P_{\mathcal{G}}(G),
\end{equation}
for all $E\in\mathcal{E}$ and all $G\in\mathcal{G}$, then $P_{\mathcal{E}\mathcal{G}}=P_{\mathcal{G}\mathcal{E}}$. We give now a simple counterexample to show that the sufficient condition does not hold for DPK.

Suppose that we observe $x_1=1$, $x_2=x_3=5$, $x_4=7$, and $x_5=8$. They induce partition $\mathcal{E}_1=\{E_j^{\mathcal{E}_1}\}_{j=1}^5$ whose elements are $E_1^{\mathcal{E}_1}=X^{-1}(1)$, $E_2^{\mathcal{E}_1}=X^{-1}(5)$, $E_3^{\mathcal{E}_1}=X^{-1}(7)$, $E_4^{\mathcal{E}_1}=X^{-1}(8)$, and $E_5^{\mathcal{E}_1}=(\cup_{j=1}^4 E_j^{\mathcal{E}_1})^c$. The empirical probabilities assigned to the elements of $\mathcal{E}_1$ according to \eqref{emp_prob_1} and \eqref{emp_prob_2} are $P^{emp}_1(E_j^{\mathcal{E}_1})=1/6$, $j\in\{1,3,4,5\}$, and $P^{emp}_1(E_2^{\mathcal{E}_1})=1/3$. Now, suppose that we observe a new data point $x_6=11$, so that $x_1,\ldots,x_6$ induce a new partition $\mathcal{E}_2=\{E_j^{\mathcal{E}_2}\}_{j=1}^6$ whose elements are such that $E_j^{\mathcal{E}_2}=E_j^{\mathcal{E}_1}$ for $j\in\{1,\ldots,4\}$,  $E_5^{\mathcal{E}_2}=X^{-1}(11)$, and $E_6^{\mathcal{E}_2}=(\cup_{j=1}^5 E_j^{\mathcal{E}_2})^c$. As we can see, $E_5^{\mathcal{E}_2}\cup E_6^{\mathcal{E}_2}=E_5^{\mathcal{E}_1}$, so $\mathcal{E}_2$ is a refinement of $\mathcal{E}_1$. The empirical probabilities assigned to the elements of $\mathcal{E}_2$ according to \eqref{emp_prob_1} and \eqref{emp_prob_2} are $P^{emp}_2(E_j^{\mathcal{E}_2})=1/7$, $j\in\{1,3,4,5,6\}$, and $P^{emp}_2(E_2^{\mathcal{E}_2})=2/7$. Then, we have that
\begin{align*}
   P_{\mathcal{E}_1 \mathcal{E}_2}(E_1^{\mathcal{E}_1})&=P_{\mathcal{E}_1 \mathcal{E}_2}(E_1^{\mathcal{E}_2})=\beta(6) P_{\mathcal{E}_1}(E_1^{\mathcal{E}_1}) + [1-\beta(6)] 1/7\\ &\neq P_{\mathcal{E}_1}(E_1^{\mathcal{E}_1})=\beta(5) P(E_1^{\mathcal{E}_1}) + [1-\beta(5)] 1/6,
\end{align*}
which does not meet condition \eqref{suff_cond_d&z}.

%$P_{\mathcal{E}_1 \mathcal{E}_2}(E_1^{\mathcal{E}_2})=P_{\mathcal{E}_1 \mathcal{E}_2}(E_3^{\mathcal{E}_2})=P_{\mathcal{E}_1 \mathcal{E}_2}(E_4^{\mathcal{E}_2})=P_{\mathcal{E}_1 \mathcal{E}_2}(E_5^{\mathcal{E}_2})=P_{\mathcal{E}_1 \mathcal{E}_2}(E_6^{\mathcal{E}_2})=1/7$, and $P_{\mathcal{E}_1 \mathcal{E}_2}(E_2^{\mathcal{E}_2})=2/7$. Then, we immediately see that $P_{\mathcal{E}_1 \mathcal{E}_2}(E_1^{\mathcal{E}_2})=P_{\mathcal{E}_1 \mathcal{E}_2}(E_1^{\mathcal{E}_1})=1/7\neq 1/6=P_{\mathcal{E}_1}(E_1^{\mathcal{E}_1})$, which does not meet condition \eqref{suff_cond_d&z}.

In \cite[Theorem 3.2]{diaconis_zabell}, the authors show that $P_{\mathcal{E}\mathcal{G}}=P_{\mathcal{G}\mathcal{E}}$ if and only if $\mathcal{E}$ and $\mathcal{G}$ are Jeffrey-independent, that is, if and only if $P_\mathcal{E}(G)=P(G)$ and $P_\mathcal{G}(E)=P(E)$, for all $E\in\mathcal{E}$ and all $G\in\mathcal{G}$. The underlying implicit assumption to this result, though, appears to be the fact that $P(E\cap G)>0$, for all $E$ and all $G$. As it is immediate to see, this does not hold in our case, so we cannot use this result to check the commutativity of DPK updates. For example, if $\mathcal{E}=\mathcal{G}$, pick any $E_1,E_2\in\mathcal{E}$, $E_1\neq E_2$. Then, $E_1\cap E_2=\emptyset$, and so $P(E_1\cap E_2)=0$.

Should the lack of commutativity worry the agent that intends to update their beliefs using DPK? The answer is no. Since successive partitions are induced by an increasing amount of collected data points, commutativity would mean that losing data yields no loss of information on the likelihood of the event $A \subset \Omega$ of interest. This is undesirable: the more we know about the composition of $\Omega$, the better we want our assessment to be on the plausibility of event $A$. As Diaconis and Zabell point out in \cite[Section 4.2, Remark 2]{diaconis_zabell}, ``noncommutativity is not a real problem for successive Jeffrey updating''; it is not a real problem for DPK either.

Before concluding this Remark, we mention how, despite DPK is not in general commutative, the limit probability $P_{\tilde{\mathcal{E}}}$ is the same regardless of the order in which data is collected. 
Suppose we collect observations in a different order in two different procedures. Call $(\mathcal{E}_t)$ and $(\mathcal{E}_t^\prime)$ the sequences of successive partitions in the first and second procedures, respectively, and $\tilde{\mathcal{E}}$ and $\tilde{\mathcal{E}}^\prime$ the limit partitions for the first and second procedures, respectively. 
%That is, they are the finest partitions compatible with the elements of $\mathcal{X}$ according to DPK. 
\begin{proposition}\label{data_order}
Suppose $\lim_{n_t\rightarrow\infty}\beta(n_t)=0$, $[1-\beta(n_t)]/n_t=O(1/n_t)$, and $\mathbb{E}(X)<\infty$. Call $P_{\tilde{\mathcal{E}}}$ the almost sure limit of $(P_{\mathcal{E}_t})$ and $P_{\tilde{\mathcal{E}}^\prime}$ the almost sure limit of $(P_{\mathcal{E}^\prime_t})$ in the total variation metric as $n_t$ goes to infinity. Then, $P_{\tilde{\mathcal{E}}}= P_{\tilde{\mathcal{E}}^\prime}$.
\end{proposition}
\end{remark}

\begin{remark}\label{path_independence}
    In this remark we discuss an appealing choice of $\beta(n_t)$. 
    %that makes DPK path independent. By \textit{path independence}, we mean that first observing $x_1,\ldots,x_{n_t}$ and then $x_{n_{t}+1},\ldots,x_{n_{t+1}}$ is equivalent to directly observing $x_1,\ldots,x_{n_{t+1}}$. This entails that $P_{\mathcal{E}_1 \cdots \mathcal{E}_t} = P_{\mathcal{E}_t}$, for all $t$.
    Let us first first describe the update from $t=0$ to $t=1$. Suppose at time $t=0$ our prior $P\equiv P_{\mathcal{E}_0}$ has \textit{confidence index} $C_0=\mathscr{K}$, where $\mathscr{K} \in \mathbb{N}$ is the prior sample size of $P$ \cite{meng_prior}. We collect observations $x_1,\ldots,x_{n_1}$, and we put
    $$\beta(n_1)=\frac{C_0}{C_0+n_1},$$
    so that 
    $$P_{E_1}(E)=\frac{C_0}{C_0+n_1}P_{\mathcal{E}_0}(E)+\frac{n_1}{C_0+n_1}P^{emp}_1(E), \quad \forall E \in\mathcal{E}_1.$$
    The confidence index is then updated to $C_1=C_0+n_1$. 
    %This procedure is entirely path independent in terms of the observations. If we first collect $\ell$ observations from $x_1,\ldots,x_{n_1}$, update $P$ and its confidence index according to these $\ell$ observations and then update the new prior and new confidence index using the remaining $n_1 - \ell$ observations, we obtain exactly the same $P_{E_1}(E)$, for all $E\in\mathcal{E}_1$ and the same updated confidence index. 
    In general, we have that  
    $$\beta(n_t)=\frac{C_{t-1}}{C_{t-1}+n_t-n_{t-1}} \quad \text{and} \quad C_t=\begin{cases}
    \mathscr{K} & \text{for } t=0\\
    C_{t-1}+n_t-n_{t-1} & \text{for } t\geq 1
    \end{cases},$$
    where $n_0=0$ by convention.\footnote{Because $\beta(n_t)$ depends on the prior sample size $\mathscr{K}$ of $P$, for all $t$, for notational clarity we should write $\beta(n_t,P)$. We do not do so to lighten the notation and to make it consistent with the rest of the paper.} Notice that $\beta(n_t)$ can be rewritten as $C_{t-1}/C_t$, so it can be expressed as the \textit{relative prior confidence}: the closer it is to $1$, the less the collected observations influence our previous opinion, and so the stickier the DPK update is. The opposite holds the closer $\beta(n_t)$ is to $0$. Notice also that $\lim_{n_t\rightarrow\infty}\beta(n_t)=0$ and $[1-\beta(n_t)]/n_t=O(1/n_t)$, so Theorem \ref{prop3} can be applied, provided that we assume $\mathbb{E}(X)<\infty.$
\end{remark}

\begin{remark}
In PK, an agent's subjective probabilities over a fixed partition undergo a change (a \textit{Jeffrey shift}), which is then propagated across the rest of their probabilities in a natural manner. Crucially, PK does not specify what Jeffrey shift an agent's probabilities will undergo; it treats the choice of the Jeffrey shift as an input to the rule rather than part of the rule itself. Indeed, in the original interpretation of PK, the shift is usually a non-inferential change to the agent's degrees of belief that is not chosen consciously or freely, but rather e.g. the brute result of a perceptual process. 

DPK is an updating technique that sits in between Bayes' and Jeffrey's rules. It can be seen, heuristically, as a map from specifications of statistical problems to choices of Jeffrey shift (which are then propagated in the usual way, via PK). While it is built as a particular case of PK, it uses the empirical distribution to assign probabilities to the elements of the updated partition $\mathcal{E}_t$. In order to mechanize the procedure, it gives up the freedom of choosing whatever probability the agent feels correct to assign to the elements of $\mathcal{E}_t$. At the same time, if evidence is collected that does not belong to $\Omega$, that is, if $\mathcal{X}\not\subset\Omega$, then using the inverse image $X^{-1}$ of function $X$, DPK allows one to update their beliefs without first needing to enlarge $\Omega$ to $\Omega^\prime=\Omega\times\mathcal{X}$. Notice also that, being a particular case of PK, DPK updates can be obtained by Bayesian updating in a larger space $\Omega^\prime$ \cite[Theorem 2.1]{diaconis_zabell}.\footnote{In \cite{diaconis_zabell}, the authors show that there exists a ``duality'' between Bayes' rule (BR) and PK. BR can be seen as a special case of PK, as we pointed out in section \ref{why}, while at the same time we can obtain PK from BR if we enlarge the state space.} There are two main reasons for not wanting to enlarge the state space:
\begin{itemize}
    \item reassessing our beliefs on a larger space requires us to extend our beliefs from the elements of $2^\Omega$ to those of $2^{\Omega^\prime}$; we can do so using Halmos' extension \cite[Exercise 48.4]{halmos}, \cite[Section 4.13]{billingsley};\footnote{For the imprecise version of DPK, that is, for DIPK, we can extend the agent's beliefs via Walley's extension \cite[Chapter 3]{walley}.}
    \item updating probabilities on a larger sigma-algebra can be computationally costly.
\end{itemize}
Besides simplifying the updating procedure by not requiring an enlarged state space, we also conjecture that DPK simplifies the treatment of nuisance parameters, a statement that will be verified in future work.

\end{remark}

\section{Working with sets of probabilities}\label{sets}
In this Section, we generalize dynamic probability kinematics to dynamic imprecise probability kinematics (DIPK). To do so, we first need to introduce the concepts of lower probability, upper probability, and core of a lower probability.

\subsection{Concepts}\label{concepts}
Consider a generic set of probabilities $\Pi$ on a measurable space $(\Omega,\mathcal{F})$. The lower probability of $A$ associated with $\Pi$ is defined as 
$$\underline{P}(A):=\inf_{P \in \Pi} P(A), \quad \forall A \in \mathcal{F}.$$
The upper probability of $A$ associated with $\Pi$ is defined as the conjugate to $\underline{P}(A)$, that is,  
$$\overline{P}(A):=1-\underline{P}(A^c)=\sup_{P^\prime \in \Pi} P^\prime(A), \quad \forall A \in \mathcal{F}.$$
Recall that $\Delta(\Omega,\mathcal{F})$ denotes the set of all probability measures on $(\Omega,\mathcal{F})$. Lower probability $\underline{P}$ completely characterizes the set
\begin{align*}
   \text{core}(\underline{P})&:=\{P \in \Delta(\Omega,\mathcal{F}): P(A) \geq \underline{P}(A), \forall A \in \mathcal{F}\}\\ &=\{P \in \Delta(\Omega,\mathcal{F}): \overline{P}(A) \geq P(A) \geq \underline{P}(A), \forall A \in \mathcal{F}\},
\end{align*}
where the second equality is a characterization \cite[Page 3389]{cerreia}. It is the set of all probability measures on $(\Omega,\mathcal{F})$ that setwise dominate $\underline{P}$. Notice that the core is convex \cite[Section 2.2]{marinacci2} and weak$^\star$-compact \cite[Proposition 3]{marinacci2}.\footnote{Recall that in the weak$^\star$ topology, a net $(P_\alpha)_{\alpha \in I}$ converges to $P$ if and only if $P_\alpha(A) \rightarrow P(A)$, for all $A \in \mathcal{F}$.} 

By complete characterization, we mean that it is sufficient to know $\underline{P}$ to be able to completely specify core$(\underline{P})$. To emphasize this aspect, some authors say that $\underline{P}$ is \textit{compatible} with core$(\underline{P})$ \cite{gong}.

To generalize DPK to DIPK, we first prescribe the agent to specify a set of probabilities $\mathcal{P}$, then to compute the lower probability associated with it. The core of such lower probability represents the agent's initial beliefs. To update their beliefs, the agent computes the DPK update of the extrema of the core, that is, of the elements of the core that cannot be written as a convex combination of other elements. Their updated beliefs are represented by the convex hull of the updated extrema, which coincides with the core of the updated lower probability by \cite[Theorem 3.6.2]{walley}.
%We work with sets of probabilities because DPK is a prior sensitive procedure, so probability sets allow for a robust analysis. 

We require the agent's beliefs to be represented by the core for two main reasons. The first, mathematical, one is to ensure that the belief set can be completely characterized by the lower probability, and that lower probability $\underline{P}$ is coherent \cite[Section 3.3.3]{walley}. The second, philosophical, one is presented in Remark \ref{rem_on_lower_prev_and_sensitivity_approach}.

\begin{remark}\label{rem_on_lower_prev_and_sensitivity_approach}
%The way the agent specifies the set of probability measures representing their initial beliefs is equivalent to them performing what is known as  \textit{sensitivity analysis}. That is, a
At the beginning of the study, the \textit{sensitivity analysis} approach to imprecise probabilities prescribes the agent to specify a set of possible (or plausible) candidates for the true or ideal probability measure $P_T$ governing the events of interest \cite{berger}. As \cite[Section 5.9]{walley} points out, this way of proceeding assumes the \textit{axiom of ideal precision}: there exists a true probability measure $P_T$ governing the random events, but it cannot be precisely known e.g. because we would need an infinitely long reflection to elicit it. 

\iffalse
Nevertheless, the sensitivity analysis approach is mathematically equivalent to:
\begin{itemize}
    \item[1.] assessing the lower probability $\underline{P}^\prime$ of some events $\{A_k\} \subset \mathcal{F}$ of interest, defined as the supremum price the agent is willing to pay to enter a bet that pays \$$1$ if event $A_k$ takes place, and \$$0$ otherwise,
    \item[2.] considering its natural extension $\underline{P}$ to all the elements in $\mathcal{F}$ \cite[Section 3.1]{walley},
    \item[3.] considering the set of regular probability measures that setwise dominate lower probability $\underline{P}$, that is, $\text{core}(\underline{P})$.
\end{itemize}
For this reason, we deem the sensitivity analysis approach satisfactory. 
\fi

The philosophical motivation for the agent's beliefs being represented by the core of $\underline{P}$ is the following. A criticism brought forward by Walley in  \cite[Section 2.10.4.(c)]{walley} is that, given a lower probability $\underline{P}$, there is no cogent reason for which the agent should choose a specific $P_T$ that dominates $\underline{P}$, or -- for that matter -- a collection of ``plausible'' probabilities. Because the core considers all (regular) probability measures that dominate $\underline{P}$, it is the perfect instrument to reconcile Walley's behavioral and sensitivity analysis interpretations.\footnote{In the imprecise probabilities literature, agents are often required to specify coherent lower (and upper) probabilities \cite[Section 2.5]{walley}. 
    %A way of doing so is to require that no Dutch books can be made against the punter \cite[Definition 5.1]{caprio}, i.e. that we cannot find a finite collection $\{B_j\}_{j=1}^n \subset \mathcal{F}$ along with numbers $\{s_j\}_{j=1}^n \subset \mathbb{R}$ such that, for all $\omega \in \Omega$, $\sum_{j=1}^n s_j \left[ \mathbbm{1}_{B_j}(\omega)-\underline{P}(B_j) \right]<0$. 
    In \cite[Section 3.3.3]{walley} the author shows that $\underline{P}$ is coherent if and only if it can be written as the infimum of a set $\mathcal{P}$ of regular probability measures.}

It is worth to notice that lower probabilities are a special case of \textit{lower previsions} \cite{troffaes,walley}. To define these latter, we need to first introduce the concept of \textit{gambles}. A gamble $Y$ is a bounded real-valued function on $\Omega$ which is interpreted as an uncertain reward. The set of all gambles on $\Omega$ is denoted by $\mathcal{L}(\Omega)$. Call now $\mathcal{K}$ an arbitrary subset of $\mathcal{L}(\Omega)$; a lower prevision $\underline{P}$ is a real-valued function defined on $\mathcal{K}$ such that, for all $K\in\mathcal{K}$, $\underline{P}(K)$ is the supremum price $\mu$ for which it is asserted that the gamble $X-\mu$ is desirable to the agent \cite[Section 2.3.1]{walley}. If we have a set $\mathcal{P}$ of probability measures, then $\underline{P}(X)=\inf_{P\in\mathcal{P}} \mathbb{E}_P(X)$. Consider now a generic event $A \in \mathcal{F}$, and call $\mathcal{A}:=\{\mathbb{I}_A(\omega):\omega\in\Omega \text{, } A\in\mathcal{F}\}$ the collection of indicator functions of events $A\in\mathcal{F}$. We can see how an indicator function is just a $0-1$ valued gamble, and so lower probabilities can be seen as lower previsions defined on $\mathcal{A}\subset\mathcal{L}(\Omega)$ \cite[Section 2.7.2]{walley}. In this work we focus on lower probabilities because they are more immediately related to regular (additive) probabilities, and because they are easier to derive from a set of probability measures. In the future, we will generalize DIPK to deal with lower previsions. 
\end{remark}

\subsection{DIPK for sets of probabilities}
The analysis begins with specifying a set $\mathcal{P} \subset \Delta(\Omega,\mathcal{F})$ of probability measures on $\Omega$. 
%(since the agent needs only to specify a finite number of probability measures)
%and causes no loss of generality since, as we shall see, we will compute the convex hull of this set. 
We then consider $\underline{P}\equiv \underline{P}_{\mathcal{E}_0}$, the lower probability associated with $\mathcal{P}$. The set representing the agent's initial beliefs is given by $\mathcal{P}^{\text{co}}_{\mathcal{E}_0}=\text{core}(\underline{P}_{\mathcal{E}_0})$, where superscript co denotes the fact that $\mathcal{P}^{\text{co}}_{\mathcal{E}_0}$ is convex and compact. The importance of these properties is explained in Remark \ref{convexity}. We also need to consider the set $\mathcal{P}_{\mathcal{E}_0}=\text{ex}\mathcal{P}^{\text{co}}_{\mathcal{E}_0}$ of extrema of $\mathcal{P}^{\text{co}}_{\mathcal{E}_0}$. Of course,  $\mathcal{P}^{\text{co}}_{\mathcal{E}_0}=\text{Conv}(\mathcal{P}_{\mathcal{E}_0})$, where $\text{Conv}(\cdot)$ denotes the convex hull.

\iffalse
$\mathcal{P}^{\text{co}}_{\mathcal{E}_0}:={\text{Conv}\mathcal{P}_{\mathcal{E}_0}}$, that is, the convex hull of $\mathcal{P}_{\mathcal{E}_0}$. This is going to represent the starting point of our analysis. We need to work with $\mathcal{P}^{\text{co}}_{\mathcal{E}_t}$ because, as we shall see, the upper and lower probabilities associated with $\mathcal{P}^{\text{co}}_{\mathcal{E}_t}$ (defined later) -- together with our updating procedure -- completely characterize $\mathcal{P}^{\text{co}}_{\mathcal{E}_t}$, but they do not completely characterize $\mathcal{P}_{\mathcal{E}_t}$, for all $t \in \mathbb{N}$.
\fi

We then compute the DPK update of every element in $\mathcal{P}_{\mathcal{E}_0}$, and we obtain $$\mathcal{P}_{\mathcal{E}_1}:=\left\lbrace{P_{\mathcal{E}_1} \in \Delta(\Omega,\mathcal{F}) : P_{\mathcal{E}_1}(A)=\sum_{E_j \in \mathcal{E}_1} P_{\mathcal{E}_0}(A \mid E_j)P_{\mathcal{E}_1}(E_j) \text{, } \forall A \in \mathcal{F}, P_{\mathcal{E}_0} \in \mathcal{P}_{\mathcal{E}_0} }\right\rbrace.$$
After that, we compute $\mathcal{P}^\text{co}_{\mathcal{E}_1}=\text{Conv}(\mathcal{P}_{\mathcal{E}_1})=\text{core}(\underline{P}_{\mathcal{E}_1})$, where $\underline{P}_{\mathcal{E}_1}$ is the updated lower probability, and the last equality holds by \cite[Theorem 3.6.2]{walley}. 

Repeating this procedure, we build two sequences, $(\mathcal{P}_{\mathcal{E}_t})$ and $(\mathcal{P}^{\text{co}}_{\mathcal{E}_t})$. Notice that for any $t \in \mathbb{N}$, the lower and upper probabilities associated with $\mathcal{P}_{\mathcal{E}_t}$ are equal to the lower and upper probabilities associated with $\mathcal{P}^\text{co}_{\mathcal{E}_t}$, respectively. An example of how to update subjective beliefs according to DIPK is given in section \ref{soccer}.

Recall that $d_{TV}$ denotes the total variation distance
$$d_{TV}(\pi,\gamma):=\sup_{A\in\mathcal{F}} \left| \pi(A)-\gamma(A) \right|,$$
for all $\pi,\gamma\in\Delta(\Omega,\mathcal{F})$. Suppose 
%that the data points $x_1,\ldots,x_{n_t}$ that induce partition $\mathcal{E}_t$ are observed independently, 
$\lim_{n_t\rightarrow\infty}\beta(n_t)=0$, $[1-\beta(n_t)]/n_t=O(1/n_t)$, and $\mathbb{E}(X)<\infty$. Call 
$$\mathcal{P}_{\tilde{\mathcal{E}}}:=\left\lbrace{P_{\tilde{\mathcal{E}}} \in \Delta(\Omega,\mathcal{F}) : d_{TV}(P_{\mathcal{E}_t}, P_{\tilde{\mathcal{E}}})\xrightarrow[n_t\rightarrow\infty]{a.s.}0 \text{, } P_{\mathcal{E}_t} \in \mathcal{P}_{\mathcal{E}_t} }\right\rbrace.$$
That is, $\mathcal{P}_{\tilde{\mathcal{E}}}$ is the set of limits (as $n_t$ goes to infinity with probability $1$ in the total variation metric) of the elements $P_{\mathcal{E}_t}$ of set $\mathcal{P}_{\mathcal{E}_t}$ representing the (extrema of the) agent's updated beliefs. We are sure $\mathcal{P}_{\tilde{\mathcal{E}}}$ is not empty by Proposition \ref{limit} and Theorem \ref{prop3}. Then, by construction, we have that 
\begin{equation}\label{hausd}
 d_H(\mathcal{P}_{\mathcal{E}_t},\mathcal{P}_{\tilde{\mathcal{E}}})=\max \left( \sup_{P \in \mathcal{P}_{\mathcal{E}_t}} d_{TV}(P,\mathcal{P}_{\tilde{\mathcal{E}}}), \sup_{P^\prime \in \mathcal{P}_{\tilde{\mathcal{E}}}} d_{TV}(\mathcal{P}_{\mathcal{E}_t},P^\prime) \right) \rightarrow 0   
\end{equation}
as $n_t$ goes to infinity with probability $1$, where $d_H$ denotes the Hausdorff metric, and, in general, $d_{TV}(\pi,\Gamma):=\inf_{\gamma \in \Gamma} d_{TV}(\pi,\gamma)$, for all $\pi \in \Delta(\Omega,\mathcal{F})$ and all $\Gamma \subset \Delta(\Omega,\mathcal{F})$. Such a convergence is true also for $\mathcal{P}^{\text{co}}_{\mathcal{E}_t}$ and $\mathcal{P}^{\text{co}}_{\tilde{\mathcal{E}}}$, as shown in the next proposition. 
\begin{proposition} \label{cvg_cl_co}
If $\lim_{n_t\rightarrow\infty}\beta(n_t)=0$, $[1-\beta(n_t)]/n_t=O(1/n_t)$, and $\mathbb{E}(X)<\infty$, then the following is true with probability $1$
$$d_H(\mathcal{P}^{\text{co}}_{\mathcal{E}_t},\mathcal{P}^{\text{co}}_{\tilde{\mathcal{E}}})\rightarrow 0$$
as $n_t$ go to infinity.
\end{proposition}

\begin{remark}\label{convexity}
\iffalse
Consider $\mathcal{P}_{\mathcal{E}_t}$, and take two of its components ${P}_{1,\mathcal{E}_t}$, ${P}_{2,\mathcal{E}_t}$. We have that $\alpha {P}_{1,\mathcal{E}_t}(A) + (1-\alpha) {P}_{2,\mathcal{E}_t}(A) \in [\underline{P}_{\mathcal{E}_t}(A), \overline{P}_{\mathcal{E}_t}(A)]$, for some $\alpha \in [0,1]$ and all $A \in \mathcal{F}$, but $\alpha {P}_{1,\mathcal{E}_t} + (1-\alpha) {P}_{2,\mathcal{E}_t}$ need not belong to $\mathcal{P}_{\mathcal{E}_t}$. To this extent, upper and lower probabilities $\overline{P}_{\mathcal{E}_t}$ and $\underline{P}_{\mathcal{E}_t}$ are not able to capture ``holes'' and ``dents'' in the set $\mathcal{P}_{\mathcal{E}_t}$. That is why we need the sequence of convex sets $(\mathcal{P}_{\mathcal{E}_t}^{\text{co}})$ to represent the agent's belief updating procedure.
\fi
Let us discuss the importance of $\mathcal{P}_{\mathcal{E}_t}^{\text{co}}$ being convex and compact. Consider a generic set of probabilities $\Pi$ on a measurable space $(\Omega,\mathcal{F})$. Suppose $\Pi$ is finite, i.e. $\Pi=\{\pi_j\}_{j=1}^k$, for some $k \in \mathbb{N}$. Then, the lower probability associated with $\Pi$ is equivalent to the one associated with its convex hull $\text{Conv}(\Pi)$. If instead $\Pi$ is convex but open, then the lower probability associated with $\Pi$ is equivalent to the one associated with its closure $\text{Cl}(\Pi)$. To this extent, lower probabilities are not able to detect ``holes and dents'' in their associated set of probabilities. This is why we need the sequence of convex and (weak$^\star$-)compact sets $(\mathcal{P}_{\mathcal{E}_t}^{\text{co}})$ to represent the agent's belief updating procedure.

We (tacitly) assumed that the extrema of $\mathcal{P}^\text{co}_{\mathcal{E}_0}$ are finite; we did so for the following reasons. 
 \begin{itemize}
     \item If $\Omega$ is finite, we can see probability measures as vectors in the unit simplex $\mathscr{S}^{d-1}$ of $\mathbb{R}^d$, where $d=\#\Omega<\infty$. So the core of a lower probability will be a (closed) convex subset of $\mathscr{S}^{d-1}$, which can be approximated arbitrarily well by a polytope having finitely many vertices \cite{bronstein}.\footnote{Here ``approximated arbitrarily well'' means that some distance between the convex set and the polytope, e.g. the Hausdorff metric, can be made arbitrarily small.} The polytope with finitely many vertices is the geometric representation of a closed and convex set of probabilities having finitely many extrema.
     \item If $\Omega$ is countable, the assumption is stronger, and has mainly a computational motivation. It corresponds to the agent specifying a \textit{finitely generated credal set} -- that is, the convex hull of finitely many probability measures -- that is (possibly) a superset of the convex hull of $P_1,\ldots,P_k$, $\text{Conv}(P_1,\ldots,P_k)$. 
 \end{itemize}
  As \cite[Lemma 13]{ergo_dipk} shows, if $\Omega$ is at most countable and we work with lower previsions instead of lower probabilities, the core (appropriately redefined) of $\underline{P}$ and $\text{Conv}(P_1,\ldots,P_k)$ coincide, so the assumption that they are equal is automatically verified. This will prove useful when we will generalize DIPK to deal with lower previsions.

%Notice that, for any $t$, if $\mathcal{P}_{\mathcal{E}_t}$ is not convex and at least closed,  lower  probability $\underline{P}_{\mathcal{E}_t}$ does not completely characterize $\mathcal{P}_{\mathcal{E}_t}$. To see this, suppose that $\underline{P}_{\mathcal{E}_t}$ is finite, i.e. $\mathcal{P}_{\mathcal{E}_t}=\{{P}_{j,\mathcal{E}_t}\}_{j=1}^k$, for some $k \in \mathbb{N}$. Then, the lower probability associated with $\mathcal{P}_{\mathcal{E}_t}$ is equivalent to the one associated with its convex hull Conv$(\mathcal{P}_{\mathcal{E}_t})$. If instead $\mathcal{P}_{\mathcal{E}_t}$ is convex but open, then the lower probability associated with $\mathcal{P}_{\mathcal{E}_t}$ is equivalent to the one associated with its closure Cl$(\mathcal{P}_{\mathcal{E}_t})$. To this extent, lower extended probabilities are not able to detect ``holes and dents'' in their associated set of extended probabilities. This is why the concept of core is so crucial and we need the sequence of convex and (weak$^\star$-)compact sets $(\mathcal{P}_{\mathcal{E}_t}^{\text{co}})$ to represent the agent's belief updating procedure.
\end{remark}

\begin{remark}\label{why_all_this_complexity}
A natural question the reader may ask is why do we need the core to represent the agent beliefs. Indeed, it would be easier to require the agent to specify a finite set of plausible probability measures, and then let the convex hull of such finite set represent their initial beliefs.\footnote{Notice that the convex hull is both convex and weak$^\star$-compact. Compactness comes from it being the convex hull of a finite set in a Banach space (the normed vector space induced by $(\Delta(\Omega,\mathcal{F}),d_{TV})$ is complete because $d_{TV}$ is a complete metric; notice also that $\|\cdot\|_{TV}$-compact implies weak$^\star$-compact by the definition of weak$^\star$-compactness).} The answer is because the lower probability completely characterizes the core, but does not completely characterize the convex hull. In general the convex hull of a finite set of probabilities is a \textit{proper subset} of the core of the lower probability associated with that set, \cite[Example 1]{amarante2} and \cite[Examples 6,7,8]{amarante}. This means that when studying the DIPK update from $\mathcal{P}^{\text{co}}_{\mathcal{E}_t}$ to  $\mathcal{P}^{\text{co}}_{\mathcal{E}_{t+1}}$  we can just update the lower probability $\underline{P}_{\mathcal{E}_t}$ to $\underline{P}_{\mathcal{E}_{t+1}}$ to be able to specify the whole $\mathcal{P}^{\text{co}}_{\mathcal{E}_{t+1}}$. This would not be the case had we not represented the agent's beliefs via the core.

\end{remark}

\begin{remark}\label{near_ign}
    Notice that equation \eqref{tacit_assumption} implies a \textit{near-ignorance} assumption in the DIPK update. This means that every element in $\mathcal{P}_{\mathcal{E}_0}=\text{ex}\mathcal{P}^{\text{co}}_{\mathcal{E}_0}$
    %the collection $\{P_1,\ldots,P_k\}$ specified by the agent at the beginning of the analysis 
    gives positive probability to all nonempty $A \in\mathcal{F}$. This is desirable because no finite sample is enough to annihilate a sufficiently extreme prior belief. To see this, suppose that there is a $P\in\mathcal{P}^{\text{co}}_{\mathcal{E}_0}$ and an $A^\prime\in \mathcal{F}$ such that $P(A^\prime)=0$; then 
\begin{itemize}
    \item $\underline{P}(A^\prime)=0$, and
    \item $P_\mathcal{E}(A^\prime)=0$ as well, since $P(A^\prime\cap E) \leq P(A^\prime)$, for all $E\in\mathcal{E}$, by the monotonicity of probability measures. This implies that $\underline{P}_\mathcal{E}(A^\prime)=0$.
\end{itemize}
As we can see, no finite amount of data can resolve vacuous initial beliefs.
\end{remark}

\section{Procedures to obtain and bound upper and lower probabilities}\label{proc}
%\subsection{The exhaustive updating procedure.}
As we have seen in Remark \ref{why_all_this_complexity}, the lower probability associated with $\mathcal{P}^{\text{co}}_{\mathcal{E}_t}$ encodes all the information contained in the set. It is natural, then, that we focus our attention on $\underline{P}_{\mathcal{E}_t}$. In this Section, given a generic $t\in\mathbb{N}$, we derive bounds for $\underline{P}_{\mathcal{E}_{t+1}}$ that can be computed without performing the DIPK updated of $\mathcal{P}^{\text{co}}_{\mathcal{E}_t}$. They are interesting in their own right, and will be put to use in Section \ref{behavior} to study the behavior of set $\mathcal{P}^{\text{co}}_{\mathcal{E}_t}$ with respect to set $\mathcal{P}^{\text{co}}_{\mathcal{E}_{t-1}}$.

For any $A \in \mathcal{F}$, and any element $E$ of a generic partition $\mathbf{\mathcal{E}}$, define 
\begin{equation*}
\underline{P}^B_{\mathcal{E}_t}(A \mid E):= \inf_{P \in \mathcal{P}^{\text{co}}_{\mathcal{E}_t}}P(A \mid E) = \inf_{P \in \mathcal{P}^{\text{co}}_{\mathcal{E}_t}} \frac{P(A \cap E)}{P(E)}
\end{equation*}
and 
\begin{equation*}
\overline{P}^B_{\mathcal{E}_t}(A \mid E):= \sup_{P \in \mathcal{P}^{\text{co}}_{\mathcal{E}_t}}P(A \mid E) = \sup_{P \in \mathcal{P}^{\text{co}}_{\mathcal{E}_t}} \frac{P(A \cap E)}{P(E)}.
\end{equation*}
These are called the generalized Bayes' conditional lower and upper probabilities \cite{wasserman}, respectively.
%, and we are going to discuss more about them later in the section. 
We have the following. 
\begin{proposition}\label{ulb}
For any $A \in \mathcal{F}$ and any $t \in \mathbb{N}$,
\begin{equation}\label{up_bd}
\underline{P}_{\mathcal{E}_{t+1}}(A) \geq \sum_{E_j \in \mathcal{E}_{t+1}}  \underline{P}^B_{\mathcal{E}_t}(A \mid E_j) \left[ \beta(n_t)\underline{P}_{\mathcal{E}_t}(E_j) + (1-\beta(n_t))P^{emp}_{t+1}(E_j) \right]
\end{equation}
and 
\begin{equation}\label{low_bd}
\overline{P}_{\mathcal{E}_{t+1}}(A) \leq \sum_{E_j \in \mathcal{E}_{t+1}}  \overline{P}^B_{\mathcal{E}_t}(A \mid E_j) \left[ \beta(n_t)\overline{P}_{\mathcal{E}_t}(E_j) + (1-\beta(n_t))P^{emp}_{t+1}(E_j) \right].
\end{equation}
\end{proposition}

\begin{corollary}\label{cor_imp}
For all $t \in \mathbb{N}$, all $P_{\mathcal{E}_{t+1}} \in \mathcal{P}^{\text{co}}_{\mathcal{E}_{t+1}}$, and all $A \in\mathcal{F}$, 
\begin{align*}
    P_{\mathcal{E}_{t+1}}(A) \in \bigg[ &\sum_{E_j \in \mathcal{E}_{t+1}}  \underline{P}^B_{\mathcal{E}_t}(A \mid E_j) \left[ \beta(n_t)\underline{P}_{\mathcal{E}_t}(E_j) + (1-\beta(n_t))P^{emp}_{t+1}(E_j) \right] ,\\ &\sum_{E_j \in \mathcal{E}_{t+1}}  \overline{P}^B_{\mathcal{E}_t}(A \mid E_j) \left[ \beta(n_t)\overline{P}_{\mathcal{E}_t}(E_j) + (1-\beta(n_t))P^{emp}_{t+1}(E_j) \right] \bigg].
\end{align*}
\end{corollary}

%The bounds we presented allow us to circumvent the bottleneck presented by step (2) in the exhaustive procedure. This because it is easier to compute $\underline{P}^B_{\mathcal{E}_t}(A \mid E_j)$ and $\overline{P}^B_{\mathcal{E}_t}(A \mid E_j)$, for all $A \in \mathcal{F}$ and all $E_j \in \mathcal{E}_{t+1}$ rather than first computing the DIPK update of  $\mathcal{P}_{\mathcal{E}_t}$ to $\mathcal{P}_{\mathcal{E}_{t+1}}$, and then calculating $\underline{P}_{\mathcal{E}_{t+1}}$ and $\overline{P}_{\mathcal{E}_{t+1}}$.

There are two other ways to define lower and upper conditional probabilities. The first one, called geometric update, is such that for any $A \in \mathcal{F}$, and any element $E$ of a generic partition $\mathbf{\mathcal{E}}$,
$$\underline{P}^G_{\mathcal{E}_t}(A \mid E):=  \frac{\inf_{P \in \mathcal{P}^\text{co}_{\mathcal{E}_t}} P(A \cap E)}{\inf_{P \in \mathcal{P}^\text{co}_{\mathcal{E}_t}} P(E)} = \frac{\underline{P}_{\mathcal{E}_t}(A \cap E)}{\underline{P}_{\mathcal{E}_t}(E)} \quad \text{and} \quad \overline{P}^G_{\mathcal{E}_t}(A \mid E)=\frac{\overline{P}_{\mathcal{E}_t}(A \cap E)}{\overline{P}_{\mathcal{E}_t}(E)}.$$
The other one, called Dempster's rule of conditioning, is the natural dual to the geometric procedure. It differs from this latter from the operational point of view \cite[Section 2]{gong}, but since mathematically they are the same, we are not going to cover Dempster's rule in the present work.

An interpretation of how generalized Bayes' and geometric rules come about when a generic partition $\{E_j\}$ of $\Omega$ is available is the following. Let $\sqcup$ denote the union of disjoint sets, and $\underline{P}$ a generic lower probability. We know that lower probabilities are superadditive, so since given any $A \in \mathcal{F}$ we have that $A=\sqcup_j (A \cap E_j)$, it follows that 
\begin{equation}\label{interpr}
\underline{P}(A) \geq \sum_j \underline{P}(A \cap E_j).
\end{equation}
Now, $\underline{P}(A \cap E_j)$ can be interpreted as the lowest possible probability attached to event $A \cap E_j$, in which case we retrieve generalized Bayes' rule. It can also be rewritten as $\frac{\underline{P}(A \cap E_j)}{\underline{P}(E_j)} \underline{P}(E_j)$; in this latter case, we retrieve the geometric rule. It is worth noting that, for any lower probability $\underline{P}$, by \cite[Lemma 5.3]{gong} we have that
\begin{equation}\label{relation}
\underline{P}^B(A \mid B) \leq \underline{P}^G(A \mid B) \leq \overline{P}^G(A \mid B) \leq \overline{P}^B(A \mid B),
\end{equation}
for all $A,B \in \mathcal{F}$.
\subsection{Geometric rule}\label{lunch}
As we have seen in Proposition \ref{ulb}, generalized Bayes comes naturally from our updating procedure. This because, as argued in Section \ref{why}, Jeffrey's rule is a generalization of Bayesian conditioning. Given the inequalities in \eqref{relation}, we can sharpen the bounds we found using generalized Bayes' rule by using the geometric rule. 
%, but this comes at the cost of an assumption. As the proverb goes, there is no free lunch. 

%Notice that, by Remark \ref{zero_denom}, we have that, for all $A\in\mathcal{F}$ and all $t\in\mathbb{N}$,
%$$\sum_{E_j\in\mathcal{E}_{t+1}} P_{\mathcal{E}_{t}}(A\mid E_j)P_{\mathcal{E}_{t+1}}(E_j).$$
%= \sum_{E_j\in\mathcal{E}_{t+1}  } P_{\mathcal{E}_{t}}(A\mid E_j)P_{\mathcal{E}_{t+1}}(E_j) since we assumed that if $P_{\mathcal{E}_{t}}(E_j)= 0$ for some $E_j\in\mathcal{E}_{t+1}$, then $P_{\mathcal{E}_{t}}(A\mid E_j)=0$.

\begin{proposition}\label{free:lunch}
For any $A \in \mathcal{F}$ and any $t \in \mathbb{N}$,
\begin{equation}\label{geom-lb}
\underline{P}_{\mathcal{E}_{t+1}}(A) \geq \sum_{E_j\in\mathcal{E}_{t+1}  }  \underline{P}^G_{\mathcal{E}_t}(A \mid E_j) \left[ \beta(n_t)\underline{P}_{\mathcal{E}_t}(E_j) + (1-\beta(n_t))P^{emp}_{t+1}(E_j) \right]
\end{equation}
and 
\begin{equation}\label{geom-ub}
\overline{P}_{\mathcal{E}_{t+1}}(A) \leq \sum_{E_j\in\mathcal{E}_{t+1}  }  \overline{P}^G_{\mathcal{E}_t}(A \mid E_j) \left[ \beta(n_t)\overline{P}_{\mathcal{E}_t}(E_j) + (1-\beta(n_t))P^{emp}_{t+1}(E_j) \right].
\end{equation}
\end{proposition}

\begin{corollary}\label{cor_imp2}
For all $t \in \mathbb{N}$, all $P_{\mathcal{E}_{t+1}} \in \mathcal{P}^{\text{co}}_{\mathcal{E}_{t+1}}$, and all $A \in\mathcal{F}$, 
\begin{align}
    P_{\mathcal{E}_{t+1}}(A) \in \bigg[ &\sum_{E_j\in\mathcal{E}_{t+1}  }  \underline{P}^G_{\mathcal{E}_t}(A \mid E_j) \left[ \beta(n_t)\underline{P}_{\mathcal{E}_t}(E_j) + (1-\beta(n_t))P^{emp}_{t+1}(E_j) \right], \nonumber\\ &\sum_{E_j\in\mathcal{E}_{t+1}  }  \overline{P}^G_{\mathcal{E}_t}(A \mid E_j) \left[ \beta(n_t)\overline{P}_{\mathcal{E}_t}(E_j) + (1-\beta(n_t))P^{emp}_{t+1}(E_j) \right] \bigg].\label{interval_imp}
\end{align}
In addition, 
\begin{align}
    \sum_{E_j\in\mathcal{E}_{t+1}  } \underline{P}_{\mathcal{E}_t}^G(A \mid E_j) &\left[ \beta(n_t)\underline{P}_{\mathcal{E}_t}(E_j) + (1-\beta(n_t))P^{emp}_{t+1}(E_j) \right] \nonumber\\ &\geq \sum_{E_j\in\mathcal{E}_{t+1}} \underline{P}_{\mathcal{E}_t}^B(A \mid E_j) \left[ \beta(n_t)\underline{P}_{\mathcal{E}_t}(E_j) + (1-\beta(n_t))P^{emp}_{t+1}(E_j) \right] \label{tight_low}
\end{align}
and
\begin{align}
    \sum_{E_j\in\mathcal{E}_{t+1}  } \overline{P}_{\mathcal{E}_t}^G(A \mid E_j) &\left[ \beta(n_t)\overline{P}_{\mathcal{E}_t}(E_j) + (1-\beta(n_t))P^{emp}_{t+1}(E_j) \right] \nonumber\\ &\leq \sum_{E_j\in\mathcal{E}_{t+1}} \overline{P}_{\mathcal{E}_t}^B(A \mid E_j) \left[ \beta(n_t)\overline{P}_{\mathcal{E}_t}(E_j) + (1-\beta(n_t))P^{emp}_{t+1}(E_j) \right]. \label{tight_up}
\end{align}
\end{corollary}

%We need to explicitly require $P_{\mathcal{E}_{t}}(E_j)\neq 0$ for a technical detail in the proof of Proposition \ref{free:lunch}. 
Corollary \ref{cor_imp2} implies that 
%\robin{Explain the feasibility to ascertain the "assumption" in Prop 8.2.}
%\robin{How about moving Corollary 8.2.1 before Prop 8.2? The corollary is true regardless whether the ratio is equal.}

%Condition $\frac{\underline{P}_{\mathcal{E}_{t+1}}(A \cap E_j)}{\underline{P}_{\mathcal{E}_{t+1}}(E_j)} = \frac{\underline{P}_{\mathcal{E}_{t}}(A \cap E_j)}{\underline{P}_{\mathcal{E}_{t}}(E_j)}$ in Proposition~\ref{free:lunch} requires that the relative lower probability assigned to $A \cap E_j$ with respect to the one assigned to $E_j$, for all $A \in \mathcal{F}$ and all $E_j$ in the updated partition, has to stay constant when we perform the DIPK update of a probability set. In a sense, this condition is an assumption governs the dynamic of the lower probability assignment. 
\begin{align*}
\bigg[ &\sum_{E_j \in \mathcal{E}_{t+1}}  \underline{P}^G_{\mathcal{E}_t}(A \mid E_j) \left[ \beta(n_t)\underline{P}_{\mathcal{E}_t}(E_j) + (1-\beta(n_t))P^{emp}_{t+1}(E_j) \right] ,\\  &\sum_{E_j \in \mathcal{E}_{t+1}}  \overline{P}^G_{\mathcal{E}_t}(A \mid E_j) \left[ \beta(n_t)\underline{P}_{\mathcal{E}_t}(E_j) + (1-\beta(n_t))P^{emp}_{t+1}(E_j) \right] \bigg] \\ 
\subset \bigg[ &\sum_{E_j \in \mathcal{E}_{t+1}}  \underline{P}^B_{\mathcal{E}_t}(A \mid E_j) \left[ \beta(n_t)\underline{P}_{\mathcal{E}_t}(E_j) + (1-\beta(n_t))P^{emp}_{t+1}(E_j) \right] ,\\
&\sum_{E_j \in \mathcal{E}_{t+1}}  \overline{P}^B_{\mathcal{E}_t}(A \mid E_j) \left[ \beta(n_t)\underline{P}_{\mathcal{E}_t}(E_j) + (1-\beta(n_t))P^{emp}_{t+1}(E_j) \right] \bigg],
\end{align*}
so we retrieve tighter bounds for $\underline{P}_{\mathcal{E}_{t+1}}(A)$ and $\overline{P}_{\mathcal{E}_{t+1}}(A)$, and also obtain a tighter interval around ${P}_{\mathcal{E}_{t+1}}(A)$, for all $A \in \mathcal{F}$ and all $t \in \mathbb{N}$.

%***DISCUSS THE ASSUMPTION?***

\section{Behavior of updated sets of probabilities}\label{behavior}
In the imprecise probabilities literature, three concepts are crucial regarding the behavior of updated sets of probabilities. They are contraction, dilation, and sure loss. In this Section, building on the definitions in \cite[Section 3]{gong}, we introduce the concepts of DIPK-contraction, DIPK-dilation and DIPK-sure loss, and we give sufficient conditions for them to take place. 
%The focus is on when they can take place, rather than how likely is that to happen. This because, as pointed out in \cite[Section 5.1]{gong}, there are updating techniques  that do not give rise to one or more such behaviors. For example, generalized Bayes rule alone cannot induce contraction, while DIPK can, as we show in Section \ref{soccer}.

Fix some $t \in \mathbb{N}$. We say that $\mathcal{P}^{\text{co}}_{\mathcal{E}_t}$ \textit{DIPK-contracts with respect to $\mathcal{P}^{\text{co}}_{\mathcal{E}_{t-1}}$ for some $A \in  \mathcal{F}$} if $\underline{P}_{\mathcal{E}_t}(A) \geq \underline{P}_{\mathcal{E}_{t-1}}(A)$ and $\overline{P}_{\mathcal{E}_t}(A) \leq \overline{P}_{\mathcal{E}_{t-1}}(A)$, and at least one inequality is strict. It \textit{strictly DIPK-contracts} if both the inequalities are strict. In addition, we say that sequence $(\mathcal{P}^{\text{co}}_{\mathcal{E}_t})$ \textit{DIPK-contracts for some $A \in  \mathcal{F}$} if for any $t \in \mathbb{N}$, we have that $\underline{P}_{\mathcal{E}_t}(A) \geq \underline{P}_{\mathcal{E}_{t-1}}(A)$, $\overline{P}_{\mathcal{E}_t}(A) \leq \overline{P}_{\mathcal{E}_{t-1}}(A)$, and the inequalities are strict for some $t$.

\textit{DIPK-dilation} is defined analogously, by inverting the inequality signs.

Finally, we say that $\mathcal{P}^{\text{co}}_{\mathcal{E}_t}$ exhibits \textit{DIPK-sure loss with respect to $\mathcal{P}^{\text{co}}_{\mathcal{E}_{t-1}}$ for some $A \in  \mathcal{F}$} if $\underline{P}_{\mathcal{E}_t}(A) > \overline{P}_{\mathcal{E}_{t-1}}(A)$ or $\overline{P}_{\mathcal{E}_t}(A) < \underline{P}_{\mathcal{E}_{t-1}}(A)$.

%Despite being similar to the classical definitions of contraction, dilation, and sure loss -- see e.g. \cite[Section 3]{gong} -- our definitions differ slightly to take into account our updating procedure. For this reason, we find results that in a classical imprecise probabilities setting may be impossible. For example, as shown in \cite[Section 5.1]{gong}, in a classical imprecise probabilities setting, generalized Bayes rule cannot contract nor induce sure loss. Instead, we have the following.      in the imprecise probabilities literature

\begin{proposition}\label{contr}
For any $t \in \mathbb{N}$, 
%let $Q_{\mathcal{E}_t}$ be the probability measure described in \eqref{eq:imp}. Then, 
sufficient conditions for $\mathcal{P}^{\text{co}}_{\mathcal{E}_t}$ to DIPK-contract with respect to $\mathcal{P}^{\text{co}}_{\mathcal{E}_{t-1}}$ for some $A \in \mathcal{F}$ are the following
$$\sum_{E_j \in \mathcal{E}_{t}}  \underline{P}^B_{\mathcal{E}_{t-1}}(A \mid E_j) \left[ \beta(n_t)\underline{P}_{\mathcal{E}_{t-1}}(E_j) + (1-\beta(n_t))P^{emp}_{t}(E_j) \right] \geq \underline{P}_{\mathcal{E}_{t-1}}(A)$$
and
$$\sum_{E_j \in \mathcal{E}_{t}}  \overline{P}^B_{\mathcal{E}_{t-1}}(A \mid E_j) \left[ \beta(n_t)\overline{P}_{\mathcal{E}_{t-1}}(E_j) + (1-\beta(n_t))P^{emp}_{t}(E_j) \right] \leq \overline{P}_{\mathcal{E}_{t-1}}(A),$$
and at least one inequality is strict.
\end{proposition}

Notice that we obtain strict DIPK-contraction if both the inequalities are strict. We have the same results if we use geometric lower conditional probabilities instead of the generalized Bayes' ones. We also have the following.
\begin{proposition}\label{sure:loss}
For any $t \in \mathbb{N}$, sufficient conditions for $\mathcal{P}^{\text{co}}_{\mathcal{E}_t}$ to exhibit DIPK-sure loss with respect to $\mathcal{P}^{\text{co}}_{\mathcal{E}_{t-1}}$ for some $A \in \mathcal{F}$ are the following
$$\sum_{E_j \in \mathcal{E}_{t}}  \underline{P}^B_{\mathcal{E}_{t-1}}(A \mid E_j) \left[ \beta(n_t)\underline{P}_{\mathcal{E}_{t-1}}(E_j) + (1-\beta(n_t))P^{emp}_{t}(E_j) \right] > \overline{P}_{\mathcal{E}_{t-1}}(A)$$
or
$$\sum_{E_j \in \mathcal{E}_{t}}  \overline{P}^B_{\mathcal{E}_{t-1}}(A \mid E_j) \left[ \beta(n_t)\overline{P}_{\mathcal{E}_{t-1}}(E_j) + (1-\beta(n_t))P^{emp}_{t}(E_j) \right] < \underline{P}_{\mathcal{E}_{t-1}}(A).$$
\end{proposition}

Again, we obtain the same conditions if we use geometric lower conditional probabilities instead of the generalized Bayes' ones.

%\robin{Note that there are sufficient conditions for contraction and sure loss because the upper and lower probabilities from the GB-Jeffreys updates are conservative bounds to the exhaustive methods. The same bounds cannot be used to sufficiently establish dilation.}

Giving a sufficient condition for DIPK-dilation without directly computing $\underline{P}_{\mathcal{E}_t}(A)$ and $\overline{P}_{\mathcal{E}_t}(A)$ is less straightforward. We have the following.
%This because, unlike DIPK-contraction and DIPK-sure loss for which the upper and lower probabilities from DIPK updating are conservative bounds to the exhaustive methods, the same bounds cannot be used to sufficiently establish DIPK-dilation. 
\begin{proposition}\label{dilation}
For any $t \in \mathbb{N}$ and some $A \in \mathcal{F}$, if there exist $P_{s,\mathcal{E}_{t}}, P_{k,\mathcal{E}_{t}} \in \mathcal{P}^{\text{co}}_{\mathcal{E}_{t}}$ 
\iffalse
such that their updates $P_{s,\mathcal{E}_{t}}, P_{k,\mathcal{E}_{t}} \in \mathcal{P}^{\text{co}}_{\mathcal{E}_{t}}$ are such that 
$$P_{s,\mathcal{E}_{t}}(A) \geq \underline{P}_{\mathcal{E}_{t}}(A) \quad \text{and} \quad P_{k,\mathcal{E}_{t}}(A) \leq \overline{P}_{\mathcal{E}_{t}}(A).$$
\fi
such that $\underline{P}_{\mathcal{E}_{t-1}}(A) \geq P_{s,\mathcal{E}_{t}}(A)$ and $\overline{P}_{\mathcal{E}_{t-1}}(A) \leq P_{k,\mathcal{E}_{t}}(A)$, then $\mathcal{P}^{\text{co}}_{\mathcal{E}_t}$ DIPK-dilates with respect to $\mathcal{P}^{\text{co}}_{\mathcal{E}_{t-1}}$ for $A$, and at least one inequality is strict.
\end{proposition}

We obtain strict DIPK-dilation if both the inequalities in Proposition \ref{dilation} are strict. As we can see, we do not need to directly compute $\underline{P}_{\mathcal{E}_t}(A)$ and $\overline{P}_{\mathcal{E}_t}(A)$. We only need to find $P_{s,\mathcal{E}_{t-1}}, P_{k,\mathcal{E}_{t-1}} \in \mathcal{P}^{\text{co}}_{\mathcal{E}_{t-1}}$ such that their updates satisfy the assumptions in Proposition \ref{dilation}.
%, and update them.
%We need the extra assumption because otherwise we do not have an upper bound for $\underline{P}_{\mathcal{E}_{t}}(A)$ or a lower bound for $\overline{P}_{\mathcal{E}_{t}}(A)$, and so giving conditions for dilation to take place becomes problematic.

We can give a result, similar to Proposition \ref{dilation} that provides sufficient conditions for $\mathcal{P}^\text{co}_{\mathcal{E}_t}$ to DIPK-contract with respect to $\mathcal{P}^\text{co}_{\mathcal{E}_{t-1}}$ for some $A \in \mathcal{F}$. This is interesting because, contrary to what we have in Proposition \ref{contr}, we do not use the notions of lower and upper conditional probabilities. Its downside is that it requires the computation of both $\underline{P}_{\mathcal{E}_t}(A)$ and $\overline{P}_{\mathcal{E}_t}(A)$.
\begin{proposition}\label{contr2}
For any $t \in \mathbb{N}$ and some $A \in \mathcal{F}$, if there exist $P_{s,\mathcal{E}_{t-1}}, P_{k,\mathcal{E}_{t-1}} \in \mathcal{P}^{\text{co}}_{\mathcal{E}_{t-1}}$ such that $\underline{P}_{\mathcal{E}_{t}}(A) \geq P_{k,\mathcal{E}_{t-1}}(A)$ and $\overline{P}_{\mathcal{E}_{t}}(A) \leq P_{s,\mathcal{E}_{t-1}}(A)$, and at least one inequality is strict, then $\mathcal{P}^{\text{co}}_{\mathcal{E}_t}$ DIPK-contracts with respect to $\mathcal{P}^{\text{co}}_{\mathcal{E}_{t-1}}$ for $A$.
\end{proposition}

We obtain strict DIPK-contraction if both the inequalities in Proposition \ref{contr2} are strict. Notice that we cannot give a result similar to Propositions \ref{dilation} and \ref{contr2} for DIPK-sure loss because we cannot require any assumption on any $P_{\mathcal{E}_{t-1}} \in \mathcal{P}^\text{co}_{\mathcal{E}_{t-1}}$, $P_{\mathcal{E}_t} \in \mathcal{P}^\text{co}_{\mathcal{E}_t}$ to make them ``sit in between'' $\underline{P}_{\mathcal{E}_t}(A)$ and $\overline{P}_{\mathcal{E}_{t-1}}(A)$, or $\underline{P}_{\mathcal{E}_{t-1}}(A)$ and $\overline{P}_{\mathcal{E}_t}(A)$.

\section{Two simple examples of DPK and DIPK updating}\label{appl}
In this Section, we present two examples on how to update subjective beliefs according to DPK and DIPK procedures.

\subsection{Trials of a new surgical procedure}\label{surgery}
We continue Example \ref{ex_diac_zab}, and show how to frame it within the DPK paradigm. Recall that we wish to form a probabilistic opinion of a new surgical procedure to be performed three times at a new hospital. Upon one colleague's suggestion that another hospital performed this type of procedure with a success rate of $0.8$, we update by considering random variable $X:\Omega \rightarrow \mathcal{X}=\{0,1,2,3\}$ whose distribution is unknown and such that $X(\omega)$ represents the number of $1$'s in $\omega$.\footnote{Since we observe realizations from the same random variable $X$, it does not make sense to talk about exchangeability of $P$ as in Example \ref{ex_diac_zab}.} As we can see, $X^{-1}(3)=\{111\}$, $X^{-1}(2)=\{011,101,110\}$, $X^{-1}(1)=\{001,010,100\}$, $X^{-1}(0)=\{000\}$. The finest partition of $\Omega$ according to DPK, then, is given by $\tilde{\mathcal{E}}=\{E_0,E_1,E_2,E_3,E_4\}$, where $E_j=X^{-1}(j)$, $j\in\{0,1,2,3\}$, and $E_4=\emptyset$. Recall that in DPK data points contribute information not through their sheer number, but rather the way the partition the space and assign relative frequencies. The information that our colleague provided us is equivalent to observing $1000$ data points $x_1,\ldots,x_{1000}$, out of which $512$ are all $3$'s, $384$ are all $2$'s, $96$ are all $1$'s, and $8$ are all $0$'s. This because the relative frequency $Fr$ of the elements of $\mathcal{X}$ is $Fr(\{3\})=512/1000=1\cdot 0.8^3$, $Fr(\{2\})=384/1000=3 \cdot 0.2 \cdot 0.8^2$,  $Fr(\{1\})=96/1000=3 \cdot 0.2^2 \cdot 0.8$, and $Fr(\{0\})=8/1000=1 \cdot 0.2^3$. But why should they be derived in this way? We have that $Fr(\{3\})=1\cdot 0.8^3$ because there is only $1$ way of obtaining three successes, each of which has probability $0.8$ in the procedures conducted at the hospital that our colleague informed us about. Instead, $Fr(\{2\})=3 \cdot 0.2 \cdot 0.8^2$ because there are $3$ ways of obtaining two successes and one failure, where the probability of the latter is $0.2$ according to our colleague. Finally, $Fr(\{1\})=3 \cdot 0.2^2 \cdot 0.8$ because there are $3$ ways of obtaining one successes and two failures, and $Fr(\{0\})=1\cdot 0.2^3$ because there is only $1$ way of obtaining three failures.

Relative frequency $Fr$ implies that $P^{emp}_1(E_0)=0.008$, $P^{emp}_1(E_1)=0.096$, $P^{emp}_1(E_2)=0.384$, $P^{emp}_1(E_3)=0.512$, and $P^{emp}_1(E_4)=0$. This corresponds to collecting the following probabilistic evidence: three failures with probability $0.008$, only one success with probability $0.096$, two successes with probability $0.384$, and three successes with probability $0.512$. We are now ready to compute the DPK update of our initial $P$. Given the composition of the sample space $\mathcal{X}$, we have that
%Recall that we assumed it to be exchangeable, so we have
\begin{align*}
    P(\{000\})&=p_0, \quad P(\{001\})=P(\{100\})=P(\{010\})=p_1, \\ P(\{110\})&=P(\{101\})=P(\{011\})=p_2, \quad P(\{111\})=p_3.
\end{align*}
Suppose $\beta(n_t)=1/n_t$; in turn we have 
\begin{align*}
    P_{\mathcal{E}_1}(\{000\})&=\frac{p_0}{P_{\mathcal{E}_0}(E_0)} \left(\frac{p_0}{1000} + \frac{999}{1000} P^{emp}_1(E_0) \right)\\&=1\cdot \left( \frac{p_0}{1000} + \frac{7.992}{1000} \right) =\frac{p_0+7.992}{1000},\\
    P_{\mathcal{E}_1}(\{001\})=P_{\mathcal{E}_1}(\{010\})=P_{\mathcal{E}_1}(\{100\})&=\frac{p_1}{P_{\mathcal{E}_0}(E_1)}\left(\frac{p_1}{1000} + \frac{999}{1000} P^{emp}_1(E_1) \right)\\&=\frac{1}{3}\cdot \left( \frac{p_1}{1000} + \frac{95.904}{1000} \right) =\frac{p_1+95.904}{3000},\\
    P_{\mathcal{E}_1}(\{011\})=P_{\mathcal{E}_1}(\{110\})=P_{\mathcal{E}_1}(\{101\})&=\frac{p_2}{P_{\mathcal{E}_0}(E_2)}\left(\frac{p_2}{1000} + \frac{999}{1000} P^{emp}_1(E_2) \right)\\&=\frac{1}{3}\cdot \left( \frac{p_2}{1000} + \frac{383.616}{1000} \right) =\frac{p_2+383.616}{3000},\\
    P_{\mathcal{E}_1}(\{111\})&=\frac{p_3}{P_{\mathcal{E}_0}(E_3)} \left(\frac{p_3}{1000} + \frac{999}{1000} P^{emp}_1(E_3) \right)\\&=1\cdot \left( \frac{p_3}{1000} + \frac{511.488}{1000} \right) =\frac{p_3+511.488}{1000}.
\end{align*}
We can see how, because of the composition of sample space $\mathcal{X}$, in the case of only one successful outcome the updated probability $P_{\mathcal{E}_1}$ assigned to $\{001\}$, $\{010\}$, and $\{100\}$ is exactly $1/3$ of the mixture between the prior and the empirical probability of $E_1$. The same is true for the case of two successful outcomes.

%we chose $P$ to be exchangeable

To generalize the DPK updating presented here to a DIPK updating involving a set $\mathcal{P}$ of probability measures representing the initial beliefs of the agent one can follow the procedure explained in Section \ref{soccer}.

\subsection{Soccer match results}\label{soccer}
This example is built on \cite[Section 4.6.1]{walley}. Let $\Omega=\{W,D,L\}$ represent the result of soccer match Juventus Turin vs Inter Milan, where $W$ denotes a win for Juventus Turin, $D$ a draw, and $L$ a loss for Juventus Turin. Let then $X:\Omega \rightarrow \mathcal{X}=\{0,1\}$, where $1$ denotes a useful result (a victory or a draw) and $0$ denotes a defeat, so $X$ can be thought of as a Bernoulli random variable with unknown parameter. It is immediate to see how the finest partition of $\Omega$ according to DPK is given by $\tilde{\mathcal{E}}=\{E_1,E_2,E_3\}$, where $E_1=\{W,D\}$, $E_2=\{L\}$, and $E_3=\emptyset$. We call $P_{\mathcal{E}_t}$ the $t$-th update of $P\equiv P_{\mathcal{E}_0}$; $P_{\tilde{\mathcal{E}}}$ denotes the limit of sequence $(P_{\mathcal{E}_t})$.\footnote{Notice that $\tilde{\mathcal{E}}$ is attained almost immediately: it is enough to observe $x_j\neq x_k$, for some $j\neq k$.}

%Although $\tilde{\mathcal{E}}$ is attained almost immediately (it is enough to observe $x_j\neq x_k$, for some $j\neq k$), we maintain notation $P_{\mathcal{E}_t}$ to denote the $t$-th update of $P\equiv P_{\mathcal{E}_0}$. $P_{\tilde{\mathcal{E}}}$ will denote the limit of sequence $(P_{\mathcal{E}_t})$. 

The data points $x_1,\ldots,x_n$ that we collect represent the outcomes of past matches. Because the two teams are well established and high-level, it is reasonable to assume that function $X$ is fixed.

Let us describe how to perform a DIPK update of subjective beliefs in this context. Let the agent specify $\mathcal{P}\subset \Delta(\Omega,\mathcal{F})$, and suppose that the lower and upper probabilities $\underline{P}\equiv\underline{P}_{\mathcal{E}_0}$ and $\overline{P}\equiv\overline{P}_{\mathcal{E}_0}$ associated with $\mathcal{P}$ are such that $\underline{P}\left( W \right)=\underline{P}\left( D \right)=0.27$, $\overline{P}\left( W \right)=\overline{P}\left( D \right)=0.52$, $\underline{P}\left( L \right)=0.21$, and  $\overline{P}\left( L \right)=0.31$.\footnote{We write $\underline{P}(\omega)$ in place of $\underline{P}(\{\omega\})$ and $\overline{P}(\omega)$ in place of $\overline{P}(\{\omega\})$, $\omega\in\{W,D,L\}$, for notational convenience.} 

A simplex representation is given in Figure \ref{fig1} where each assessment is represented by a line parallel to one side of the simplex.\footnote{Notice that the higher the values assigned by $P$ to $\{\omega\} \subset \Omega$, the closer the line representing $P(\{\omega\})$ is to vertex $\omega\in\{W,D,L\}$.} The initial beliefs of the agent are encapsulated in $\mathcal{P}^{\text{co}}_{\mathcal{E}_0}=\text{core}(\underline{P})$. To update $\mathcal{P}^{\text{co}}_{\mathcal{E}_0}$ we need to find $\mathcal{P}_{\mathcal{E}_0}=\text{ex}\mathcal{P}^{\text{co}}_{\mathcal{E}_0}$. This is an easy job; it is sufficient to 
\begin{enumerate}
    \item equate $P(\omega)$ to either $\underline{P}(\omega)$ or $\overline{P}(\omega)$ for two of the three events. The probability of the third is then determined;
    \item check which of the resulting $P$ satisfies $\underline{P}\leq P \leq \overline{P}$.
\end{enumerate}
This procedure gives us four extreme points $\mathcal{P}_{\mathcal{E}_0}=\{P^{ex}_{1,\mathcal{E}_0},P^{ex}_{2,\mathcal{E}_0},P^{ex}_{3,\mathcal{E}_0},P^{ex}_{4,\mathcal{E}_0}\}$ such that
\begin{align*}
  (P^{ex}_{1,\mathcal{E}_0}(W),P^{ex}_{1,\mathcal{E}_0}(D),P^{ex}_{1,\mathcal{E}_0}(L))&=(0.52,0.27,0.21),\\
  (P^{ex}_{2,\mathcal{E}_0}(W),P^{ex}_{2,\mathcal{E}_0}(D),P^{ex}_{2,\mathcal{E}_0}(L))&=(0.27,0.42,0.31), \\
  (P^{ex}_{3,\mathcal{E}_0}(W),P^{ex}_{3,\mathcal{E}_0}(D),P^{ex}_{3,\mathcal{E}_0}(L))&=(0.42,0.27,0.31),\\
  (P^{ex}_{4,\mathcal{E}_0}(W),P^{ex}_{4,\mathcal{E}_0}(D),P^{ex}_{4,\mathcal{E}_0}(L))&=(0.27,0.52,0.21).
\end{align*}
The extrema $\mathcal{P}_{\mathcal{E}_0}$ of $\mathcal{P}^{\text{co}}_{\mathcal{E}_0}$ are the vertices of the grey trapezoid in Figure \ref{fig1}.

\begin{figure}[h!]
\centering
\includegraphics[width=.7\textwidth]{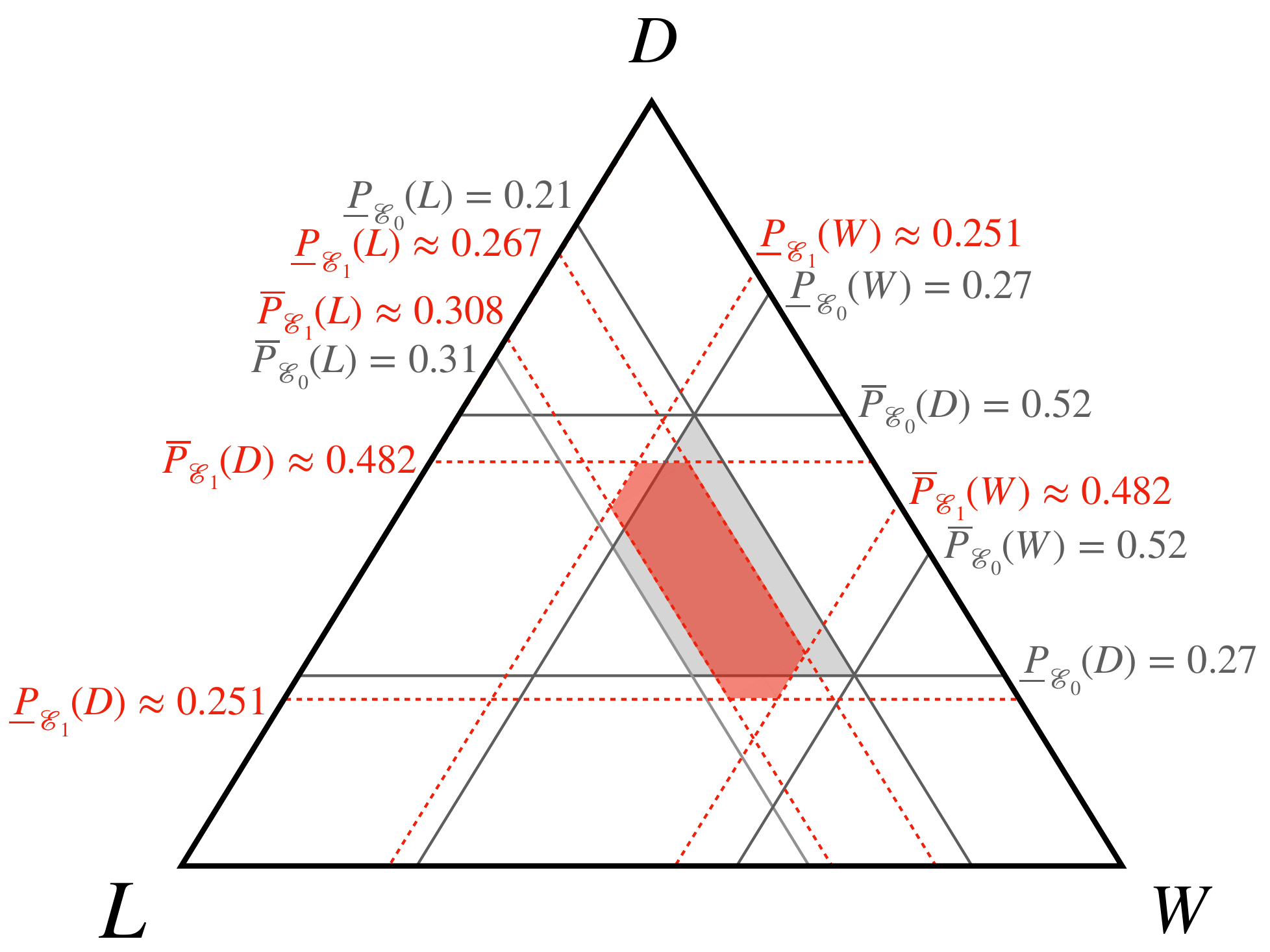}
\caption{Visual representation of $\mathcal{P}^{\text{co}}_{\mathcal{E}_0}$ (the grey trapezoid) and of $\mathcal{P}^{\text{co}}_{\mathcal{E}_1}$ (the red hexagon) in our soccer example. $\underline{P}_{\mathcal{E}_0}$ is represented by the solid grey lines, while $\underline{P}_{\mathcal{E}_1}$ by the dashed red lines.}
\centering
\label{fig1}
\end{figure}

As of January 12, 2022, there have been $257$ matches between the two teams, with $178$ useful results for Juventus Turin and $79$ wins for Inter Milan.\footnote{Data available \href{https://en.wikipedia.org/wiki/Derby_d\%27Italia}{here}.} This is to say that we observe $x_1,\ldots,x_{257}$ such that $178$ are $1$'s, and $79$ are $0$'s. Then, to compute $\mathcal{P}^{\text{co}}_{\mathcal{E}_1}$ it is enough to update the extrema in $\mathcal{P}_{\mathcal{E}_0}$ so to obtain $\mathcal{P}_{\mathcal{E}_1}$, and then consider the convex hull of the latter. The partition induced by the collected data is $\mathcal{E}_1=\{E_1,E_2,E_3\}$, and we have that $P^{emp}_1(E_1)=178/257$, $P^{emp}_1(E_2)=79/257$ and $P^{emp}_1(E_3)=0$. This corresponds to collecting the following probabilistic evidence: Juventus Turin obtains a useful result with probability $178/257$, and it loses with probability $79/257$. Let us update $P^{ex}_{1,\mathcal{E}_0}$ to $P^{ex}_{1,\mathcal{E}_1}$. Suppose $\beta(n_t)=\frac{1}{\log(n_t+1)}$;\footnote{Notice that in this example $[1-\beta(n_t)]/n_t$ is not $O(1/n_t)$, so one of the hypotheses of Proposition \ref{cvg_cl_co} is not met.} we have
\begin{align*}
    P^{ex}_{1,\mathcal{E}_1}(W)&=\frac{P^{ex}_{1,\mathcal{E}_0}(W)}{P^{ex}_{1,\mathcal{E}_0}(E_1)} P^{ex}_{1,\mathcal{E}_1}(E_1)=\frac{0.52}{0.52+0.27}\left( \frac{0.52+0.27}{\log(258)} + \frac{\log(258)-1}{\log(258)}\cdot\frac{178}{257} \right) \approx 0.482,\\
    P^{ex}_{1,\mathcal{E}_1}(D)&=\frac{P^{ex}_{1,\mathcal{E}_0}(D)}{P^{ex}_{1,\mathcal{E}_0}(E_1)} P^{ex}_{1,\mathcal{E}_1}(E_1)=\frac{0.27}{0.52+0.27}\left( \frac{0.52+0.27}{\log(258)} + \frac{\log(258)-1}{\log(258)}\cdot\frac{178}{257} \right) \approx 0.251,\\
    P^{ex}_{1,\mathcal{E}_1}(L)&=\frac{P^{ex}_{1,\mathcal{E}_0}(L)}{P^{ex}_{1,\mathcal{E}_0}(E_2)} P^{ex}_{1,\mathcal{E}_1}(E_2)=1 \cdot \left( \frac{0.21}{\log(258)} + \frac{\log(258)-1}{\log(258)}\cdot\frac{79}{257} \right) \approx 0.267,
\end{align*}
so
$$(P^{ex}_{1,\mathcal{E}_1}(W),P^{ex}_{1,\mathcal{E}_1}(D),P^{ex}_{1,\mathcal{E}_1}(L)) \approx (0.482,0.251,0.267).$$
The other elements of $\mathcal{P}_{\mathcal{E}_0}$ are updated similarly. In particular,
\begin{align*}
  (P^{ex}_{2,\mathcal{E}_1}(W),P^{ex}_{2,\mathcal{E}_1}(D),P^{ex}_{2,\mathcal{E}_1}(L))&\approx(0.271,0.421,0.308), \\
  (P^{ex}_{3,\mathcal{E}_1}(W),P^{ex}_{3,\mathcal{E}_1}(D),P^{ex}_{3,\mathcal{E}_1}(L))&\approx(0.421,0.271,0.308), \\
  (P^{ex}_{4,\mathcal{E}_1}(W),P^{ex}_{4,\mathcal{E}_1}(D),P^{ex}_{4,\mathcal{E}_1}(L))&\approx(0.251,0.482,0.267).
\end{align*}
So we have that $\underline{P}_{\mathcal{E}_1}(W)\approx 0.251\approx \underline{P}_{\mathcal{E}_1}(D)$, $\overline{P}_{\mathcal{E}_1}(W)\approx 0.482\approx \overline{P}_{\mathcal{E}_1}(D)$,  $\underline{P}_{\mathcal{E}_1}(L)\approx 0.267$, and  $\overline{P}_{\mathcal{E}_1}(L) \approx 0.308$. As we can see from Figure \ref{fig1}, the graphical representation of $\mathcal{P}^{\text{co}}_{\mathcal{E}_1}$ is a hexagon (in red). Notice also that, since $0.267 \approx \underline{P}_{\mathcal{E}_1}(L) > \underline{P}_{\mathcal{E}_0}(L)=0.21$ and $0.308 \approx \overline{P}_{\mathcal{E}_1}(L) < \overline{P}_{\mathcal{E}_0}(L)=0.31$, we have that $\mathcal{P}^{\text{co}}_{\mathcal{E}_1}$ exhibits DIPK-contraction with respect to $\mathcal{P}^{\text{co}}_{\mathcal{E}_0}$ for $L$.

\section{Conclusion}\label{concl}
In this paper, we presented dynamic probability kinematics (DPK) and dynamic imprecise probability kinematics (DIPK). These methods dynamically update subjective beliefs stated in terms of precise and imprecise probabilities, in the presence of partial information (both DPK and DIPK) and of ambiguity (DIPK only). In the case of DIPK, we provided bounds for the upper and lower probabilities associated with the updated sets, and studied their set-specific behavior including contraction, dilation, and sure loss. Two examples were provided to illustrate the procedures.

This work is just the first step towards a fully developed DIPK theory. In the future, we plan to relax the assumption that $\Omega$ needs to be at most countable. We also plan to find sufficient conditions for the inequalities in Section \ref{proc} to hold with equality. For example, in \cite{wasserman}, the authors study a Bayes' theorem for lower probabilities. They first find a lower bound for the lower posterior $\underline{P}_y$ coming from a generalization of Bayes' rule combining lower prior $\underline{P}$ with likelihood $f(y\mid \theta)$. They then show that if lower prior $\underline{P}$ is convex, that is, if $\underline{P}(A\cup B)+\underline{P}(A\cap B) \geq \underline{P}(A)+\underline{P}(B)$, then the lower bound for lower posterior $\underline{P}_y$ holds with equality. We conjecture that convexity, possibly together with additional requirements, will allow us to reach our goal.

Furthermore, we aim to generalize DIPK by allowing the agent to gather inconsistent evidence as in \cite{marchetti}. We also intend to let partial information be modeled via a set of probability distributions on $\mathcal{X}$, as empirical probabilities usually need a very large number of observations to estimate probabilities which are very close to zero or one to a good standard of relative accuracy. After that, we plan to propose a way of performing statistical analysis based on DIPK updating. Our last goal is to generalize DIPK to work with lower previsions in place of lower probabilities.

\section*{Acknowledgements}
The authors thank Sayan Murkherjee for inspiring this project and helpful discussions, Teddy Seidenfeld for a fruitful discussion about the commutativity of dynamic probability kinematics, Xiao-Li Meng for an insightful dialogue on information theory, and Alessandro Zito for his help with Figure \ref{fig1}. Research of Ruobin Gong is supported in part by the National Science Foundation (DMS-1916002). Michele Caprio would like to acknowledge partial funding by the National Science Foundation (CCF-1934964) and the Army Research Office (ARO MURI W911NF2010080).

\appendix
\section{Proofs}\label{proofs}
\begin{proof}[Proof of Proposition \ref{claim1}]
We begin by showing that $P_\mathcal{E}$ is a probability measure. We verify the Kolmogorovian axioms for a probability measure. First, we have that $P_{\mathcal{E}}(A) \geq 0$, for all $A \in \mathcal{F}$. This comes by its definition, since it is defined as the summation of products of nonnegative quantities. Second, we have that $P_{\mathcal{E}}(\Omega)=1$. This comes from the following
$$P_{\mathcal{E}}(\Omega)=\sum_{E_j \in \mathcal{E}} P(\Omega \mid E_j) P_{\mathcal{E}}(E_j)=\sum_{E_j \in \mathcal{E}}P_{\mathcal{E}}(E_j)=1.$$ 
Finally, we have that if $\{A_i\}_{i \in I}$ is a countable, pairwise disjoint collection of events, then $P_{\mathcal{E}}(\cup_{i \in I} A_i)=\sum_{i \in I} P_{\mathcal{E}}(A_i)$. This because
\begin{align*}
P_{\mathcal{E}} \left( \cup_{i \in I} A_i \right) &= \sum_{E_j \in \mathcal{E}} P \left( \cup_{i \in I} A_i \mid E_j \right) P_{\mathcal{E}}(E_j)\\
&= \sum_{E_j \in \mathcal{E}} \frac{P \left(\left[ \cup_{i \in I} A_i \right] \cap E_j\right)}{P(E_j)} P_{\mathcal{E}}(E_j)\\
&= \sum_{E_j \in \mathcal{E}} \frac{P \left( \cup_{i \in I} \left[A_i  \cap E_j\right] \right)}{P(E_j)} P_{\mathcal{E}}(E_j)\\
&= \sum_{E_j \in \mathcal{E}} \frac{ \sum_{i \in I} P \left(A_i  \cap E_j\right)}{P(E_j)} P_{\mathcal{E}}(E_j)\\
&= \sum_{i \in I} \sum_{E_j \in \mathcal{E}} \frac{ P \left(A_i  \cap E_j\right)}{P(E_j)} P_{\mathcal{E}}(E_j) = \sum_{i \in I} P_{\mathcal{E}} \left( A_i \right).
\end{align*}

We now show that $P_\mathcal{E}$ is a Jeffrey's posterior for $P$. We use \cite[Theorem 2.1]{diaconis_zabell}: it states that $P^\star$ is a Jeffrey's posterior for $P$ if and only if there exists a constant $B \geq 1$ such that $P^\star(\{\omega\}) \leq BP(\{\omega\})$, for all $\omega \in \Omega$. Fix any $\omega \in \Omega$. We have that $P_\mathcal{E}(\{\omega\}) = \sum_{E_j \in \mathcal{E}} P(\{\omega\} \mid E_j) P_\mathcal{E}(E_j)$. Call ${E}_\omega$ the element in $\mathcal{E}$ such that $\{\omega\} \cap {E}_\omega \neq \emptyset$. Then, we have that 
\begin{align}\label{eq_om}
P_\mathcal{E}(\{\omega\})= \sum_{E_j \in \mathcal{E}} P(\{\omega\} \mid E_j) P_\mathcal{E}(E_j) = \frac{P_\mathcal{E}({E}_\omega)}{P({E}_\omega)}P(\{\omega\}).
\end{align}
%The fact that we assumed $P$ to assign non-zero probabilities to every element of partition $\mathcal{E}$ guarantees that ratio $\frac{P_\mathcal{E}({E}_\omega)}{P({E}_\omega)}$ does not explode. 
Now, let $B_\omega:= \lceil \frac{P_\mathcal{E}({E}_\omega)}{P({E}_\omega)} +1 \rceil$. We have that $P_\mathcal{E}(\{\omega\}) < B_\omega P(\{\omega\})$. Consider then the well-ordered collection $\{B_\omega\}_{\omega \in \Omega}$. If we let $B:=\sup_{B^\prime_\omega \in \{B_\omega\}} B^\prime_\omega$, we conclude that $P_\mathcal{E}(\{\omega\}) < B P(\{\omega\})$, for all $\omega \in \Omega$.
\end{proof}

\begin{proof}[Proof of Proposition \ref{limit}]
We have two cases. If $\cup_{i \in \mathbb{N}}x_i=\mathcal{X}$, then, since we observed all the elements of $\mathcal{X}$, and given the procedure in Sections \ref{sec:1} and \ref{mech} to refine the partition, it is immediate to see that the partition $\tilde{\mathcal{E}}$ induced by $\{x_i\}_{i\in\mathbb{N}}$ cannot be further refined. If instead $\cup_{i \in \mathbb{N}}x_i=\mathcal{X}_{reduced}\subsetneq\mathcal{X}$, then the elements of partition $\tilde{\mathcal{E}}$ will be the unique elements of the collection $\{X^{-1}(x_i)\}_{x_i\in\mathcal{X}_{reduced}}$, plus an extra one given by $(\cup_{x_i\in\mathcal{X}_{reduced}}X^{-1}(x_i))^c$.
%If $\emptyset \in \tilde{\mathcal{E}}$, this means that $\cup_{E_j \in \tilde{\mathcal{E}} \backslash \{\emptyset\}} E_j =\Omega$, which implies that $\cup_j x_j=\mathcal{X}$, where $\{x_j\}_j$ is the collection of data points we gathered. Then, since we observed all the elements of $\mathcal{X}$, and given the procedure in Sections \ref{sec:1} and \ref{mech} to refine the partition, it is immediate to see that $\tilde{\mathcal{E}}$ cannot be further refined.
%and given that observing new data points drives the refinement of the partition, that in turn drives the update of the probability measure, 
\end{proof}

\begin{proof}[Proof of Theorem \ref{prop3}]
Let $t=1$ and fix any $A\in\mathcal{F}$. Let $\#\mathcal{E}_1=m+1$ and assume without loss of generality that $E_{m+1}=\emptyset$. We have that 
\begin{align*}
    P_{\mathcal{E}_1}(A)&=\sum_{E_j\in\mathcal{E}_1} P(A\mid E_j) P_{\mathcal{E}_1}(E_j)\\
    &=\sum_{E_j\in\mathcal{E}_1} P(A\mid E_j) \left[ \beta(n_1) P(E_j) + \left( 1-\beta(n_1) \right) \frac{1}{n_1} \sum_{i=1}^{n_1} \mathbb{I}(E_j=E_i)\right].
\end{align*}
Let then $n_1\rightarrow\infty$; we have
\begin{align}
    \lim_{n_1\rightarrow\infty} P_{\mathcal{E}_1}(A)&=\lim_{n_1\rightarrow\infty}\sum_{E_j\in\mathcal{E}_1} P(A\mid E_j) \left[ \beta(n_1) P(E_j) + \left( 1-\beta(n_1) \right) \frac{1}{n_1} \sum_{i=1}^{n_1} \mathbb{I}(E_j=E_i)\right] \nonumber\\
    &=\sum_{E_j\in\tilde{\mathcal{E}}} \left\lbrace{P(A\mid E_j)\left[ \lim_{n_1\rightarrow\infty}\beta(n_1) P(E_j) + \lim_{n_1\rightarrow\infty} \frac{1-\beta(n_1)}{n_1} \sum_{i=1}^{n_1} \mathbb{I}(E_j=E_i) \right]}\right\rbrace \nonumber \\
    &= \sum_{E_j\in\tilde{\mathcal{E}}} P(A\mid E_j) Q(E_j) \label{equazione_finale}
\end{align}
with $Q$-probability $1$. The equality in \eqref{equazione_finale} comes from our assumptions and the strong law of large numbers. We considered $t=1$ to highlight the dependence of the limiting distribution on the prior $P$. For a generic $t\in\mathbb{N}$, we have that
\begin{equation}\label{general_equation}
    \lim_{n_t\rightarrow\infty} P_{\mathcal{E}_t}(A)=\sum_{E\in\tilde{\mathcal{E}}} P_{\mathcal{E}_{t-1}}(A\mid E) Q(E)
\end{equation}
almost surely, for all $A\in\mathcal{F}$. Notice that $P_{\mathcal{E}_{t-1}}(A\mid E)$ does not depend on $n_t$, and $P_{\mathcal{E}_{t-1}}$ ``contains'' the prior as shown in equation \eqref{how_2_deps_on_prior}. We denote $P_{\tilde{\mathcal{E}}}(A):=\sum_{E\in\tilde{\mathcal{E}}} P_{\mathcal{E}_{t-1}}(A\mid E) Q(E)$, for all $A\in\mathcal{F}$. It is immediate to see that $P_{\tilde{\mathcal{E}}} \in \mathscr{Q}$. Finally, notice that \eqref{general_equation} entails that $\lim_{n_t\rightarrow\infty} d_{TV}(P_{\mathcal{E}_t},P_{\tilde{\mathcal{E}}})=0$ almost surely, concluding the proof.
\end{proof}

\begin{proof}[Proof of Proposition \ref{data_order}]
We first point out that $\tilde{\mathcal{E}}=\tilde{\mathcal{E}}^\prime$. This because, no matter the order in which we collect data points $x_i \in \mathcal{X}$, in the limit we either end up observing all the elements of $\mathcal{X}$, or all the elements of $\mathcal{X}_{reduced}$ in the case $\cup_{i\in\mathbb{N}}x_i = \mathcal{X}_{reduced}\subsetneq \mathcal{X}$. So if $\tilde{\mathcal{E}}$ is finer than $\tilde{\mathcal{E}}^\prime$, this means that there exists an $\omega$ that is mapped by $X$ into two different values, a contradiction. If instead $\tilde{\mathcal{E}}$ is coarser than $\tilde{\mathcal{E}}^\prime$, this means that $\tilde{\mathcal{E}}$ can be further refined, which contradicts Proposition \ref{limit}. Then, the claim follows by the uniqueness of the limit of a sequence.
\end{proof}

\begin{proof}[Proof of Proposition \ref{cvg_cl_co}]
Fix any $t\in\mathbb{N}$, and let $\mathcal{P}_{\mathcal{E}_t}=\{\check{P}_{k,\mathcal{E}_t}\}$. Pick any $P_{\mathcal{E}_t} \in \mathcal{P}^{\text{co}}_{\mathcal{E}_t}$. Then, by the convexity of $\mathcal{P}^{\text{co}}_{\mathcal{E}_t}$, there exists a collection $\{\alpha_k\}\subset \mathbb{R}$ such that $\#\{\alpha_k\}=\#\mathcal{P}_{\mathcal{E}_t}$,  $\sum_k \alpha_k=1$, and $P_{\mathcal{E}_t}(A)=\sum_k \alpha_k \check{P}_{k,\mathcal{E}_t}(A)$, for all $A\in\mathcal{F}$. By construction and Theorem \ref{prop3}, given our assumptions we know that for all $k$, 
$$d_{TV}(\check{P}_{k,\mathcal{E}_t},\check{P}_{k,\tilde{\mathcal{E}}})\rightarrow 0$$
as $n_t$ goes to infinity with probability $1$, where $\mathcal{P}_{\tilde{\mathcal{E}}}=\{\check{P}_{k,\tilde{\mathcal{E}}}\}$. So we can conclude that there is $P_{\tilde{\mathcal{E}}} \in \mathcal{P}^{\text{co}}_{\tilde{\mathcal{E}}}$ such that $P_{\tilde{\mathcal{E}}}(A)=\sum_k \alpha_k \check{P}_{k,\tilde{\mathcal{E}}}(A)$, for all $A\in\mathcal{F}$, and $d_{TV}(P_{\mathcal{E}_t},P_{\tilde{\mathcal{E}}})\rightarrow 0$ as $n_t$ goes to infinity with $Q$-probability $1$. 

That is to say that for every element $P_{\mathcal{E}_t}$ of $\mathcal{P}^{\text{co}}_{\mathcal{E}_t}$, there is an element $P_{\tilde{\mathcal{E}}}$ of $\mathcal{P}^{\text{co}}_{\tilde{\mathcal{E}}}$ that $P_{\mathcal{E}_t}$ converges to (with probability $1$ in the total variation metric). This immediately implies that the Hausdorff distance between $\mathcal{P}^{\text{co}}_{\mathcal{E}_t}$ and $\mathcal{P}^{\text{co}}_{\tilde{\mathcal{E}}}$ goes to $0$ as $n_t$ goes to infinity with probability $1$.
\end{proof}

\begin{proof}[Proof of Proposition \ref{ulb}]
Fix any $A \in \mathcal{F}$ and any $t \in \mathbb{N}$. Notice that 
$$\underline{P}_{\mathcal{E}_{t+1}}(A):=\inf_{P_{\mathcal{E}_{t+1}} \in \mathcal{P}^{\text{co}}_{\mathcal{E}_{t+1}}} P_{\mathcal{E}_{t+1}}(A) = \inf_{\substack{P_{\mathcal{E}_{t+1}} \in \mathcal{P}^{\text{co}}_{\mathcal{E}_{t+1}} \\ P_{\mathcal{E}_t} \in \mathcal{P}^{\text{co}}_{\mathcal{E}_t}}} P_{\mathcal{E}_{t+1}}(A).$$
Then, we have that 
%, by \eqref{eq_solita_2},
\begin{align}
\underline{P}_{\mathcal{E}_{t+1}}(A) &= \inf_{\substack{P_{\mathcal{E}_{t+1}} \in \mathcal{P}^{\text{co}}_{\mathcal{E}_{t+1}} \\ P_{\mathcal{E}_t} \in \mathcal{P}^{\text{co}}_{\mathcal{E}_t}}} \sum_{E_j \in \mathcal{E}_{t+1}} {P}_{\mathcal{E}_{t}}(A \mid E_j) P_{\mathcal{E}_{t+1}}(E_j) \nonumber \\
&\geq \sum_{E_j \in \mathcal{E}_{t+1}}  \inf_{\substack{P_{\mathcal{E}_{t+1}} \in \mathcal{P}^{\text{co}}_{\mathcal{E}_{t+1}} \\ P_{\mathcal{E}_t} \in \mathcal{P}^{\text{co}}_{\mathcal{E}_t}}} \left[ {P}_{\mathcal{E}_{t}}(A \mid E_j) P_{\mathcal{E}_{t+1}}(E_j) \right] \label{ineq} \\
&\geq \sum_{E_j \in \mathcal{E}_{t+1}}  \inf_{\substack{P_{\mathcal{E}_{t+1}} \in \mathcal{P}^{\text{co}}_{\mathcal{E}_{t+1}} \\ P_{\mathcal{E}_t} \in \mathcal{P}^{\text{co}}_{\mathcal{E}_t}}}  {P}_{\mathcal{E}_{t}}(A \mid E_j) \inf_{\substack{P_{\mathcal{E}_{t+1}} \in \mathcal{P}^{\text{co}}_{\mathcal{E}_{t+1}} \\ P_{\mathcal{E}_t} \in \mathcal{P}^{\text{co}}_{\mathcal{E}_t}}} P_{\mathcal{E}_{t+1}}(E_j)  \label{ineq2} \\
&= \sum_{E_j \in \mathcal{E}_{t+1}} \inf_{\substack{P_{\mathcal{E}_{t+1}} \in \mathcal{P}^{\text{co}}_{\mathcal{E}_{t+1}} \\ P_{\mathcal{E}_t} \in \mathcal{P}^{\text{co}}_{\mathcal{E}_t}}}  {P}_{\mathcal{E}_{t}}(A \mid E_j)  \inf_{\substack{P_{\mathcal{E}_{t+1}} \in \mathcal{P}^{\text{co}}_{\mathcal{E}_{t+1}} \\ P_{\mathcal{E}_t} \in \mathcal{P}^{\text{co}}_{\mathcal{E}_t}}} \left[ \beta(n_t){P}_{\mathcal{E}_t}(E_j) + (1-\beta(n_t))P^{emp}_{t+1}(E_j) \right]  \label{equality} \\
&= \sum_{E_j \in \mathcal{E}_{t+1}} \inf_{\substack{P_{\mathcal{E}_{t+1}} \in \mathcal{P}^{\text{co}}_{\mathcal{E}_{t+1}} \\ P_{\mathcal{E}_t} \in \mathcal{P}^{\text{co}}_{\mathcal{E}_t}}}  {P}_{\mathcal{E}_{t}}(A \mid E_j)   \left[ \beta(n_t)\inf_{\substack{P_{\mathcal{E}_{t+1}} \in \mathcal{P}^{\text{co}}_{\mathcal{E}_{t+1}} \\ P_{\mathcal{E}_t} \in \mathcal{P}^{\text{co}}_{\mathcal{E}_t}}}{P}_{\mathcal{E}_t}(E_j) + (1-\beta(n_t))P^{emp}_{t+1}(E_j) \right]  \nonumber \\
&= \sum_{E_j \in \mathcal{E}_{t+1}} \underline{P}^B_{\mathcal{E}_t}(A \mid E_j) \left[ \beta(n_t)\underline{P}_{\mathcal{E}_t}(E_j) + (1-\beta(n_t))P^{emp}_{t+1}(E_j) \right]. \nonumber
\end{align}
The inequality in \eqref{ineq} comes from the well known fact that the sum of the infima is at most equal to the infimum of the sum. The inequality in \eqref{ineq2} comes from the fact that for differentiable functions, the product of the infima is at most equal to the infimum of the product. Equation \eqref{equality} comes from equation \eqref{emp_prob_3}. 
\iffalse
the fact that $$\inf_{\substack{P_{\mathcal{E}_{t+1}} \in \mathcal{P}^{\text{co}}_{\mathcal{E}_{t+1}} \\ P_{\mathcal{E}_t} \in \mathcal{P}^{\text{co}}_{\mathcal{E}_t}}} P_{\mathcal{E}_{t+1}}(E_j) = \inf_{P_{\mathcal{E}_{t+1}} \in \mathcal{P}^{\text{co}}_{\mathcal{E}_{t+1}}} P_{\mathcal{E}_{t+1}}(E_j) = \underline{P}_{\mathcal{E}_{t+1}}(E_j).$$
Finally, equation \eqref{excl} comes from \eqref{eq:inf}. 
\fi
A similar argument -- together with the facts that the supremum of the sum is at most equal to the sum of the suprema, and that for differentiable functions, the supremum of the product is at most equal to the  product of the suprema -- gives us the stated upper bound for $\overline{P}_{\mathcal{E}_{t+1}}(A)$. 
\end{proof}

\begin{proof}[Proof of Corollary \ref{cor_imp}]
Immediate from Proposition \ref{ulb} and the definitions of upper and lower probabilities.
\end{proof}

\iffalse
&=\sum_{E_j \in \mathcal{E}_{t+1} } \inf_{\substack{P_{\mathcal{E}_{t+1}} \in \mathcal{P}^{\text{co}}_{\mathcal{E}_{t+1}} \\ P_{\mathcal{E}_t} \in \mathcal{P}^{\text{co}}_{\mathcal{E}_t}}}  {P}_{\mathcal{E}_{t}}(A \mid E_j)  \inf_{\substack{P_{\mathcal{E}_{t+1}} \in \mathcal{P}^{\text{co}}_{\mathcal{E}_{t+1}} \\ P_{\mathcal{E}_t} \in \mathcal{P}^{\text{co}}_{\mathcal{E}_t}}} \left[ \beta(n_t){P}_{\mathcal{E}_t}(E_j) + (1-\beta(n_t))P^{emp}_{t+1}(E_j) \right]\label{equiv2}\\
\fi

\begin{proof}[Proof of Proposition \ref{free:lunch}]
Pick any $A \in \mathcal{F}$ and any $t \in \mathbb{N}$. Then, we have that
\begin{align}
\underline{P}&_{\mathcal{E}_{t+1}}(A) \geq \sum_{E_j \in \mathcal{E}_{t+1}} \inf_{\substack{P_{\mathcal{E}_{t+1}} \in \mathcal{P}^{\text{co}}_{\mathcal{E}_{t+1}} \\ P_{\mathcal{E}_t} \in \mathcal{P}^{\text{co}}_{\mathcal{E}_t}}}  {P}_{\mathcal{E}_{t}}(A \mid E_j)  \inf_{\substack{P_{\mathcal{E}_{t+1}} \in \mathcal{P}^{\text{co}}_{\mathcal{E}_{t+1}} \\ P_{\mathcal{E}_t} \in \mathcal{P}^{\text{co}}_{\mathcal{E}_t}}} \left[ \beta(n_t){P}_{\mathcal{E}_t}(E_j) + (1-\beta(n_t))P^{emp}_{t+1}(E_j) \right] \label{equiv29} \\
&= \sum_{E_j \in \mathcal{E}_{t+1} } \inf_{\substack{P_{\mathcal{E}_{t+1}} \in \mathcal{P}^{\text{co}}_{\mathcal{E}_{t+1}} \\ P_{\mathcal{E}_t} \in \mathcal{P}^{\text{co}}_{\mathcal{E}_t}}}  {P}_{\mathcal{E}_{t}}(A \mid E_j)  \left[ \beta(n_t)\underline{P}_{\mathcal{E}_t}(E_j) + (1-\beta(n_t))P^{emp}_{t+1}(E_j) \right] \nonumber \\
&= \sum_{E_j \in \mathcal{E}_{t+1} } \inf_{\substack{P_{\mathcal{E}_{t+1}} \in \mathcal{P}^{\text{co}}_{\mathcal{E}_{t+1}} \\ P_{\mathcal{E}_t} \in \mathcal{P}^{\text{co}}_{\mathcal{E}_t}}}  \frac{{P}_{\mathcal{E}_{t}}(A \cap E_j)}{{P}_{\mathcal{E}_{t}}(E_j)}  \left[ \beta(n_t)\underline{P}_{\mathcal{E}_t}(E_j) + (1-\beta(n_t))P^{emp}_{t+1}(E_j) \right] \nonumber\\
&\geq \sum_{E_j \in \mathcal{E}_{t+1} }   \frac{\underline{P}_{\mathcal{E}_{t}}(A \cap E_j)}{\underline{P}_{\mathcal{E}_{t}}(E_j)}  \left[ \beta(n_t)\underline{P}_{\mathcal{E}_t}(E_j) + (1-\beta(n_t))P^{emp}_{t+1}(E_j) \right] \label{equiv3}\\
&= \sum_{E_j \in \mathcal{E}_{t+1} }   \underline{P}^G_{\mathcal{E}_{t}}(A \mid E_j)  \left[ \beta(n_t)\underline{P}_{\mathcal{E}_t}(E_j) + (1-\beta(n_t))P^{emp}_{t+1}(E_j) \right]. \nonumber
\end{align}
The inequality in \eqref{equiv29} comes from \eqref{equality}. The inequality in \eqref{equiv3} comes from the fact that for differentiable functions, the product of the infima is at most equal to the infimum of the product. In particular, 
$$\inf_{\substack{P_{\mathcal{E}_{t+1}} \in \mathcal{P}^{\text{co}}_{\mathcal{E}_{t+1}} \\ P_{\mathcal{E}_t} \in \mathcal{P}^{\text{co}}_{\mathcal{E}_t}}}  {P}_{\mathcal{E}_{t}}(A \cap E_j) \frac{1}{{P}_{\mathcal{E}_{t}}(E_j)} \geq \underline{P}_{\mathcal{E}_t}(A\cap E_j) \frac{1}{\underline{P}_{\mathcal{E}_t}(E_j)},$$
for all $A\in\mathcal{F}$, all $E_j\in \mathcal{E}_{t+1}$, and all $t\in\mathbb{N}$. A similar argument gives us the stated upper bound for $\overline{P}_{\mathcal{E}_{t+1}}(A)$.
%Notice that differentiability of ${1}/{{P}_{\mathcal{E}_{t}}(E_j)}$ is guaranteed by summing over the $E_j$'s in $\mathcal{E}_{t+1}$ having nonzero $P_{\mathcal{E}_t}$-probability.  
\end{proof}

\iffalse
\sum_{E_j \in \mathcal{E}_{t+1}} \underline{P}_{\mathcal{E}_{t+1}}(A \cap E_j) \label{first_ineq}\\
&= \sum_{E_j \in \mathcal{E}_{t+1}} \frac{\underline{P}_{\mathcal{E}_{t+1}}(A \cap E_j)}{\underline{P}_{\mathcal{E}_{t+1}}(E_j)}\underline{P}_{\mathcal{E}_{t+1}}(E_j) \nonumber\\
&= \sum_{E_j \in \mathcal{E}_{t+1}} \frac{\underline{P}_{\mathcal{E}_t}(A \cap E_j)}{\underline{P}_{\mathcal{E}_t}(E_j)}\underline{P}_{\mathcal{E}_{t+1}}(E_j) \label{ass}\\
&= \sum_{E_j \in \mathcal{E}_{t+1}} \underline{P}_{\mathcal{E}_t}^G(A \mid E_j) \underline{P}_{\mathcal{E}_{t+1}}(E_j) \label{geom}\\
&= \sum_{E_j \in \mathcal{E}_{t+1}} \underline{P}_{\mathcal{E}_t}^G(A \mid E_j) P^{emp}_{t+1}(E_j). \label{q:solita}

The inequality in \eqref{first_ineq} comes from equation \eqref{interpr}. Equation \eqref{ass} comes from our assumption. Equation \eqref{geom} comes from the definition of geometric rule's lower conditional probability. Equation \eqref{q:solita} comes from \eqref{eq:inf}.
\fi

\begin{proof}[Proof of Corollary \ref{cor_imp2}]
The interval in \eqref{interval_imp} comes from inequalities \eqref{geom-lb} and \eqref{geom-ub}. The inequalities in \eqref{tight_low} and \eqref{tight_up} come from the inequalities in \eqref{relation}. 
\end{proof}

\begin{proof}[Proof of Proposition \ref{contr}]
By Proposition \ref{ulb}, we have that 
$$\underline{P}_{\mathcal{E}_{t}}(A) \geq \sum_{E_j \in \mathcal{E}_{t}}  \underline{P}^B_{\mathcal{E}_{t-1}}(A \mid E_j) \left[ \beta(n_t)\underline{P}_{\mathcal{E}_{t-1}}(E_j) + (1-\beta(n_t))P^{emp}_{t}(E_j) \right],$$
so if $$\sum_{E_j \in \mathcal{E}_{t}}  \underline{P}^B_{\mathcal{E}_{t-1}}(A \mid E_j) \left[ \beta(n_t)\underline{P}_{\mathcal{E}_{t-1}}(E_j) + (1-\beta(n_t))P^{emp}_{t}(E_j) \right] \geq \underline{P}_{\mathcal{E}_{t-1}}(A),$$then $\underline{P}_{\mathcal{E}_t}(A) \geq \underline{P}_{\mathcal{E}_{t-1}}(A)$. A similar reasoning gives us that 
$$\sum_{E_j \in \mathcal{E}_{t}}  \overline{P}^B_{\mathcal{E}_{t-1}}(A \mid E_j) \left[ \beta(n_t)\overline{P}_{\mathcal{E}_{t-1}}(E_j) + (1-\beta(n_t))P^{emp}_{t}(E_j) \right] \leq \overline{P}_{\mathcal{E}_{t-1}}(A)$$ 
implies $\overline{P}_{\mathcal{E}_t}(A) \leq \overline{P}_{\mathcal{E}_{t-1}}(A)$. In turn, we obtain the desired DIPK-contraction if at least one inequality is strict.
\end{proof}

\begin{proof}[Proof of Proposition \ref{sure:loss}]
By Proposition \ref{ulb}, we have that 
$$\underline{P}_{\mathcal{E}_{t}}(A) \geq \sum_{E_j \in \mathcal{E}_{t}}  \underline{P}^B_{\mathcal{E}_{t-1}}(A \mid E_j) \left[ \beta(n_t)\underline{P}_{\mathcal{E}_{t-1}}(E_j) + (1-\beta(n_t))P^{emp}_{t}(E_j) \right],$$
so if $$\sum_{E_j \in \mathcal{E}_{t}}  \underline{P}^B_{\mathcal{E}_{t-1}}(A \mid E_j) \left[ \beta(n_t)\underline{P}_{\mathcal{E}_{t-1}}(E_j) + (1-\beta(n_t))P^{emp}_{t}(E_j) \right] > \overline{P}_{\mathcal{E}_{t-1}}(A),$$
then $\underline{P}_{\mathcal{E}_t}(A) > \overline{P}_{\mathcal{E}_{t-1}}(A)$. A similar reasoning gives us that
$$\sum_{E_j \in \mathcal{E}_{t}}  \overline{P}^B_{\mathcal{E}_{t-1}}(A \mid E_j) \left[ \beta(n_t)\overline{P}_{\mathcal{E}_{t-1}}(E_j) + (1-\beta(n_t))P^{emp}_{t}(E_j) \right] < \underline{P}_{\mathcal{E}_{t-1}}(A)$$ 
implies $\overline{P}_{\mathcal{E}_t}(A) < \underline{P}_{\mathcal{E}_{t-1}}(A)$. In turn, we obtain the desired DIPK-sure loss.
\end{proof}

\begin{proof}[Proof of Proposition \ref{dilation}]
Fix any $t \in \mathbb{N}$ and consider some $A \in \mathcal{F}$. By the definitions of lower and upper probabilities, we have that $P_{s,\mathcal{E}_{t}}(A^\prime), P_{k,\mathcal{E}_{t}}(A^\prime) \in [\underline{P}_{\mathcal{E}_{t}}(A^\prime),\overline{P}_{\mathcal{E}_{t}}(A^\prime)]$, for all $A^\prime \in \mathcal{F}$. Then, if our hypotheses hold, we have that, for the set $A$ we have chosen, 
$$\underline{P}_{\mathcal{E}_{t-1}}(A) \geq P_{s,\mathcal{E}_{t}}(A) \geq \underline{P}_{\mathcal{E}_{t}}(A)$$
and 
$$\overline{P}_{\mathcal{E}_{t-1}}(A) \leq P_{k,\mathcal{E}_{t}}(A) \leq \overline{P}_{\mathcal{E}_{t}}(A).$$
This concludes the proof.
\iffalse
We first notice that $P_{s,\mathcal{E}_{t-1}}$ and $P_{k,\mathcal{E}_{t-1}}$ always exist. Since $\underline{P}_{\mathcal{E}_{t}}$ is the lower probability associated with $\mathcal{P}^{\text{co}}_{\mathcal{E}_t}$, we can always find a probability measure $P_{s,\mathcal{E}_{t}}$ in $\mathcal{P}^{\text{co}}_{\mathcal{E}_t}$ that assigns a value to $A$ that is greater than or equal to the one assigned to $A$ by $\underline{P}_{\mathcal{E}_{t}}$. Then, $P_{s,\mathcal{E}_{t-1}}$ is the ``prior'' of $P_{s,\mathcal{E}_{t}}$ under our updating procedure. A similar line of argument -- conducted with respect to upper probability $\overline{P}_{\mathcal{E}_{t}}$ -- shows that $P_{k,\mathcal{E}_{t-1}}$ always exists.

So we have that $P_{s,\mathcal{E}_{t}}(A)$ is an upper bound for $\underline{P}_{\mathcal{E}_{t}}(A)$. In addition, such upper bound is assumed to be smaller than or equal to $\underline{P}_{\mathcal{E}_{t-1}}(A)$, which in turn gives us that $\underline{P}_{\mathcal{E}_{t-1}}(A) \geq \underline{P}_{\mathcal{E}_{t}}(A)$. A similar reasoning gives us that $\overline{P}_{\mathcal{E}_{t-1}}(A) \leq P_{k,\mathcal{E}_{t}}(A)$ implies $\overline{P}_{\mathcal{E}_{t-1}}(A) \leq \overline{P}_{\mathcal{E}_{t}}(A)$. In turn, we obtain the desired dilation.
\fi
\end{proof}

\begin{proof}[Proof of Proposition \ref{contr2}]
Fix any $t \in \mathbb{N}$ and consider some $A \in \mathcal{F}$. By the definitions of lower and upper probabilities, we have that $P_{s,\mathcal{E}_{t-1}}(A^\prime), P_{k,\mathcal{E}_{t-1}}(A^\prime) \in [\underline{P}_{\mathcal{E}_{t-1}}(A^\prime),\overline{P}_{\mathcal{E}_{t-1}}(A^\prime)]$, for all $A^\prime \in \mathcal{F}$. Then, if our hypotheses hold, we have that, for the set $A$ we have chosen, 
$$\underline{P}_{\mathcal{E}_{t}}(A) \geq P_{k,\mathcal{E}_{t-1}}(A) \geq \underline{P}_{\mathcal{E}_{t-1}}(A)$$
and 
$$\overline{P}_{\mathcal{E}_{t}}(A) \leq P_{s,\mathcal{E}_{t-1}}(A) \leq \overline{P}_{\mathcal{E}_{t-1}}(A).$$
This concludes the proof.
\end{proof}

\bibliographystyle{plain}
\bibliography{ergodic_theory} 

\begin{thebibliography}{10}

\bibitem{amarante2}
Massimiliano Amarante and Fabio Maccheroni.
\newblock When an event makes a difference.
\newblock {\em Theory and Decision}, 60:119--126, 2006.

\bibitem{amarante}
Massimiliano Amarante, Fabio Maccheroni, Massimo Marinacci, and Luigi
  Montrucchio.
\newblock Cores of non-atomic market games.
\newblock {\em International Journal of Game Theory}, 34:399--424, 2006.

\bibitem{padding}
Mihir Bellare and Phillip Rogaway.
\newblock Optimal asymmetric encryption.
\newblock In Alfredo~De Santis, editor, {\em Advances in Cryptology ---
  {EUROCRYPT'94}}, pages 92--111. Springer Berlin Heidelberg, 1995.

\bibitem{tabia}
Salem Benferhat, Karim Tabia, and Karima Sedki.
\newblock {J}effrey’s rule of conditioning in a possibilistic framework: an
  analysis of the existence and uniqueness of the solution.
\newblock {\em Annals of Mathematics and Artificial Intelligence}, 61:185--202,
  2011.

\bibitem{berger}
James~O. Berger.
\newblock The robust {B}ayesian viewpoint.
\newblock In Joseph~B. Kadane, editor, {\em Robustness of Bayesian Analyses}.
  Amsterdam : North-Holland, 1984.

\bibitem{berger4}
James~O. Berger.
\newblock {\em {Statistical Decision Theory}}.
\newblock Springer Series in Statistics. Springer, New York, 2nd edition, 1985.

\bibitem{berger5}
James~O. Berger and L.~Mark Berliner.
\newblock Robust {B}ayes and empirical {B}ayes analysis with
  $\epsilon$-contaminated priors.
\newblock {\em The Annals of Statistics}, 14(2):461--486, 1986.

\bibitem{billingsley}
Patrick Billingsley.
\newblock {\em {Probability and Measure}}.
\newblock New York : Wiley, second edition, 1995.

\bibitem{blume}
Lawrence Blume, Adam Brandenburger, and Eddie Dekel.
\newblock Lexicographic probabilities and choice under uncertainty.
\newblock {\em Econometrica}, 59:61--79, 1991.

\bibitem{bronstein}
Efim~M. Bronstein.
\newblock Approximation of convex sets by polytopes.
\newblock {\em Journal of Mathematical Sciences}, 153(6):727--762, 2008.

\bibitem{ergo_dipk}
Michele Caprio and Sayan Mukherjee.
\newblock Ergodic theorems for dynamic imprecise probability kinematics.
\newblock {\em International Journal of Approximate Reasoning}, 152:325--343,
  2023.

\bibitem{cerreia}
Simone Cerreia-Vioglio, Fabio Maccheroni, and Massimo Marinacci.
\newblock {Ergodic theorems for lower probabilities}.
\newblock {\em Proceedings of the American Mathematical Society},
  144:3381--3396, 2015.

\bibitem{coletti}
Giulianella Coletti and Romano Scozzafava.
\newblock {\em Probabilistic Logic in a Coherent Setting}.
\newblock Trends in Logic. Dordrecht : Springer, 2002.

\bibitem{diaconis_zabell}
Persi Diaconis and Sandy~L. Zabell.
\newblock {Updating subjective probability}.
\newblock {\em Journal of the American Statistical Association},
  77(380):822--830, 1982.

\bibitem{ellsberg}
Daniel Ellsberg.
\newblock Risk, ambiguity, and the {S}avage axioms.
\newblock {\em The Quarterly Journal of Economics}, 75(4):643--669, 1961.

\bibitem{gong}
Ruobin Gong and Xiao-Li Meng.
\newblock Judicious judgment meets unsettling updating: dilation, sure loss,
  and {S}impson's paradox.
\newblock {\em Statistical Science}, 36(2):169--190, 2021.

\bibitem{halmos}
Paul~R. Halmos.
\newblock {\em Measure Theory}.
\newblock Graduate Texts in Mathematics. New York, NY : Springer, 1950.

\bibitem{huber3}
Peter~J. Huber.
\newblock The use of {C}hoquet capacities in statistics.
\newblock {\em Bulletin of the International Statistical Institute},
  4(45):181--191, 1973.

\bibitem{huber}
Peter~J. Huber and Elvezio~M. Ronchetti.
\newblock {\em Robust statistics}.
\newblock Wiley Series in Probability and Statistics. Hoboken, New Jersey :
  Wiley, 2nd edition, 2009.

\bibitem{ichihashi}
Hidetomo Ichihashi and Hideo Tanaka.
\newblock Jeffrey-like rules of conditioning for the {D}empster-{S}hafer theory
  of evidence.
\newblock {\em International Journal of Approximate Reasoning}, 3(2):143--156,
  1989.

\bibitem{jeffrey3}
Richard~C. Jeffrey.
\newblock {\em {Contributions to the Theory of Inductive Probability}}.
\newblock PhD Thesis, Princeton University, Dept. of Philosophy, 1957.

\bibitem{jeffrey}
Richard~C. Jeffrey.
\newblock {\em {The Logic of Decision}}.
\newblock Chicago : University of Chicago Press, 1965.

\bibitem{jeffrey2}
Richard~C. Jeffrey.
\newblock {Probable knowledge}.
\newblock In Imre Lakatos, editor, {\em The Problem of Inductive Logic},
  volume~51 of {\em Studies in Logic and the Foundations of Mathematics}, pages
  166 -- 190. Elsevier, 1968.

\bibitem{lewis}
David Lewis.
\newblock Probabilities of conditionals and conditional probabilities.
\newblock {\em The Philosophical Review}, 85(3):297--315, 1976.

\bibitem{tang2}
Hexin Lv, Ning Qiu, and Yongchuan Tang.
\newblock Updating probabilistic knowledge using imprecise and uncertain
  evidence.
\newblock In {\em Third International Conference on Natural Computation (ICNC
  2007)}, volume~4, pages 624--628, 2007.

\bibitem{ma}
Jianbing Ma, Weiru Lu, Didier Dubois, and Henri Prade.
\newblock Bridging {J}effrey's rule, {AGM} revision and {D}empster conditioning
  in the theory of evidence.
\newblock {\em International Journal on Artificial Intelligence Tools},
  20(4):691--720, 2011.

\bibitem{marchetti}
Sabina Marchetti and Alessandro Antonucci.
\newblock Imaginary kinematics.
\newblock In Amir Globerson and Ricardo Silva, editors, {\em {UAI} 2018:
  Proceedings of the Thirty-Fourth Conference on Uncertainty in Artificial
  Intelligence}, pages 104--113, Monterey, California, {USA}, 2018. AUAI Press.

\bibitem{marinacci2}
Massimo Marinacci and Luigi Montrucchio.
\newblock {\em {Introduction to the Mathematics of Ambiguity}}.
\newblock {Uncertainty in Economic Theory}. New York : Routledge, 2004.

\bibitem{meng_prior}
Matthew Reimherr, Xiao‐Li Meng, and Dan~L. Nicolae.
\newblock {Prior sample size extensions for assessing prior impact and
  prior‐likelihood discordance}.
\newblock {\em Journal of the Royal Statistical Society Series B},
  83(3):413--437, 2021.

\bibitem{shafer}
Glenn Shafer.
\newblock {Jeffrey's rule of conditioning}.
\newblock {\em Philosophy of Science}, 48(3):337--362, 1981.

\bibitem{skulj}
Damjan {\v S}kulj.
\newblock Jeffrey's conditioning rule in neighbourhood models.
\newblock {\em International Journal of Approximate Reasoning}, 42(3):192--211,
  2006.

\bibitem{smets}
Philippe Smets.
\newblock Jeffrey's rule of conditioning generalized to belief functions.
\newblock In {\em Proceedings of the Ninth international conference on
  Uncertainty in Artificial Intelligence}, pages 500--505, 1993.

\bibitem{tang1}
Yongchuan Tang, Shouqian Sun, and Zhongyang Li.
\newblock Conditional evidence theory and its application in knowledge
  discovery.
\newblock {\em Lecture Notes in Computer Sciences}, 3007:500--505, 2004.

\bibitem{troffaes}
Matthias C.~M. Troffaes and Gert de~Cooman.
\newblock {\em Lower Previsions}.
\newblock Wiley Series in Probability and Statistics. New York : Wiley, 2014.

\bibitem{walley}
Peter Walley.
\newblock {\em {Statistical Reasoning with Imprecise Probabilities}}, volume~42
  of {\em Monographs on Statistics and Applied Probability}.
\newblock London : Chapman and Hall, 1991.

\bibitem{wasserman}
Larry~A. Wasserman and Joseph~B. Kadane.
\newblock Bayes' theorem for {C}hoquet capacities.
\newblock {\em The Annals of Statistics}, 18(3):1328--1339, 1990.

\bibitem{zadeh}
Lotfi Zadeh.
\newblock Fuzzy sets as a basis for a theory of possibility.
\newblock {\em Fuzzy sets and systems}, 1:3--28, 1978.

\end{thebibliography}

\end{document}